\documentclass[11pt]{amsart}
\usepackage{amssymb}
\usepackage{color}
\newtheorem{theo}{Theorem} 

\newtheorem{exttheo}{Theorem} 

\newtheorem{lemma}{Lemma}[section]
\newtheorem{prop}[lemma]{Proposition}
\newtheorem{corol}[lemma]{Corollary}
\newtheorem{claim}[lemma]{Claim}

\theoremstyle{remark}
\newtheorem{remark}[lemma]{Remark}
\newtheorem*{remark*}{Remark}

\theoremstyle{definition}

\newcommand{\lf}{\left}
\newcommand{\rg}{\right}

\newcommand{\CC}{\mathbb{C}}
\newcommand{\NN}{\mathbb{N}}
\newcommand{\RR}{\mathbb{R}}

\newcommand{\eps}{\varepsilon}

\newcommand{\LL}{\text{\L}}

\newcommand{\BBB}{\mathcal{B}}

\newcommand{\HHH}{\mathcal{H}}

\newcommand{\LLL}{\mathcal{L}}

\newcommand{\OOO}{\mathcal{O}}
\newcommand{\PPP}{\mathcal{P}}

\newcommand{\cR}{\mathbb{R}}

\newcommand{\cC}{\mathbb{C}}

\DeclareMathOperator{\re}{Re}
\DeclareMathOperator{\im}{Im}

\DeclareMathOperator{\vect}{span}

\newcommand{\ds}{\displaystyle}

\numberwithin{equation}{section} 

\title[Strichartz norm estimates for NLS]
{Maximizers for the Strichartz norm for small solutions of mass-critical NLS}
\author{Thomas Duyckaerts}
\email{thomas.duyckaerts@u-cergy.fr}
\address{University of Cergy-Pontoise, Department of Mathematics, CNRS, UMR
8088, F-95000 Cergy-Pontoise, France}
\author{Frank Merle}
\email{frank.merle@u-cergy.fr}
\address{University of Cergy-Pontoise, Department of Mathematics, CNRS, 
UMR 8088, F-95000 Cergy-Pontoise, France and IHES, France}
\author{Svetlana Roudenko}
\email{svetlana@math.asu.edu}
\address{Arizona State University, Tempe, AZ 85287, USA}
\date{December 18, 2009}
\begin{document}

\begin{abstract}
Consider the mass-critical nonlinear Schr\"odinger equations in both focusing and defocusing cases for initial data in $L^2$ in space dimension $N$. By Strichartz inequality, solutions to the corresponding linear problem belong to a global $L^p$ space in the time and space variables, where $p=2+\frac{4}{N}$. In $1D$ and $2D$, the best constant for the Strichartz inequality was computed by D.~Foschi who has also shown that the maximizers are the solutions with Gaussian initial data.

Solutions to the nonlinear problem with small initial data in $L^2$ are globally defined and belong to the same global $L^p$ space. In this work we show that the maximum of the $L^p$ norm is attained for a given small mass. 
In addition, in $1D$ and $2D$, we show that the maximizer is unique and obtain a precise estimate of the maximum. In order to prove this we show that the maximum for the linear problem in $1D$ and $2D$ is  nondegenerated.
\end{abstract}

\maketitle

\section{Introduction}

We study the $L^2$-critical nonlinear Schr\"odinger (NLS) equation
in space dimension $N\geq 1$: 
\begin{equation}
 \label{CP}
\left\{
\begin{gathered}
i\partial_t u + \frac12 \,\Delta u+\gamma \,|u|^{\frac4{N}}u=0,
\quad (t,x)\in \RR\times \RR^N, \\
u_{\restriction t=0}=f\in L^2(\RR^N).
\end{gathered}
\right.
\end{equation}
We will consider both focusing
($\gamma=+1$) and defocusing ($\gamma=-1$) equations.

Let us first recall some properties of the linear problem:
\begin{equation}
 \label{CPlin}
 i\partial_t u+\frac12\Delta u=0,\quad u_{\restriction t=0}=f.
 \end{equation}
Denote by $u=e^{i \frac{t}2 \Delta} f$ the solution to \eqref{CPlin}. The mass $\|u(t)\|_{L^2}^2$ of the solution is conserved. Solutions to the linear problem satisfy the Strichartz inequality (see \cite{St77a}):
\begin{equation}
\label{strichartz}
\forall f\in L^2,\quad \left\|e^{i\frac{t}{2}\Delta}f\right\|_{L^{\frac 4N+2}_{t,x}}\leq C\|f\|_{L^2},
\end{equation} 
where 
$$
\|u\|_{L^{\frac4N+2}_{t,x}} = \left( \iint_{\RR\times\RR^N}
|u(t,x)|^{\frac4N+2}\,dt\,dx\right)^{\frac{1}{\frac 4N+2}}.
$$
By standard profile decomposition arguments, one can easily show that the maximum for the Strichartz inequality is attained. The best constant and maximizers for the Strichartz estimates were computed by D.~ Foschi \cite{Foschi07} (see also \cite{HuZh06} for another proof) for $N=1,2$. Before stating this result, we first recall some symmetries of the equations \eqref{CP} and \eqref{CPlin}.

The following group of transformations leaves the solutions
invariant under the nonlinear and linear Schr\"odinger evolution. If
$\left\{\theta_0,\rho_0,t_0,\xi_0,x_0\right\}\in
\RR\times(0,+\infty)\times \RR\times \RR^N\times \RR^N$, then if $u$
is a solution to \eqref{CP} (respectively \eqref{CPlin}), so is
\begin{equation}
 \label{invariances}
e^{i\theta_0} \rho_0^{\frac{N}{2}} e^{ix\cdot\xi_0} e^{-i\frac{t}{2}
\left| \xi_0 \right|^2} u \left(\rho_0^2 t + t_0, \rho_0 \left(x -
\frac{t}{2} \xi_0\right) + x_0\right).
\end{equation}
This includes phase invariance, scaling, time-translation, Galilean
transformation and space-translation. Another transformation of \eqref{CP} and \eqref{CPlin} is
the pseudo-conformal inversion (see \cite{Talanov70}):
\begin{equation}
\label{PC}   
\frac1{t^{N/2}} \, \exp\left(\frac{i|x|^2}{2t}\right) \,
u\left(-\frac1{t},\frac{x}{t}\right).
\end{equation} 
Note that all the preceding transformations leave the mass and the $L^{\frac{4}{N}+2}_{t,x}$ norm of the solutions invariant. The linear equation is of course also invariant under the
multiplication by a scalar: if $u(t,x)$ is a solution, so is $c_0\, u(t,x)$, $c_0 \in \cR$.

Consider the following normalized Gaussian:
\begin{equation*}
G_0(x)=\frac{1}{\pi^{N/4}}\, e^{-\frac{|x|^2}{2}}, \quad
\text{thus,} \quad \int_{\RR^N} |G_0|^2dx=1,
\end{equation*}
and its linear evolution:
\begin{equation}
\label{def_G} G(t,x)=e^{\frac12 i t \Delta} G_0 =
\frac{1}{\pi^{N/4}} \frac1{(1+it)^{N/2}} \,
e^{-\frac{|x|^2}{2(1+it)}}.
\end{equation}

\begin{exttheo}[Foschi]
\label{T:Foschi} For all $f\in L^2(\RR^N)$, $N=1, 2$,
\begin{equation*}
\left\|e^{i\frac t2 \Delta} f
\right\|_{L^{\frac4N+2}_{t,x}}^{\frac4N+2} \leq
 C_S \, \|f\|_{L^2(\cR^N)}^{\frac4N+2}, \qquad C_S =
\begin{cases}
{\frac {1}{\sqrt 3}}, & N=1\\
~{\frac 12}, &N=2.
\end{cases}
\end{equation*}
Furthermore, the equality holds if and only if $e^{i\frac t2
\Delta} f$ is, up to the symmetries \eqref{invariances} of the equation, one of the solutions
$c_0 G$, $c_0\in\CC$.
\end{exttheo}
Let us mention that the effect of the pseudo-conformal transformation \eqref{PC} on $G$ may be expressed only
with the invariances \eqref{invariances} and we can omit it from consideration in Theorem \ref{T:Foschi}.

The Strichartz estimate \eqref{strichartz} is the key ingredient to prove that the Cauchy problem \eqref{CP} is locally wellposed in $L^2$ (see \cite{CaWe90}). For small data, the solution is also globally wellposed and
the global $L^{\frac4N+2}_{t,x}$ norm is finite, which implies that the solution scatters in $L^2$. This was extended to large radial data in the defocusing case $\gamma=-1$, in \cite{TaViZh07} for $N\geq 3$  and in \cite{KiTaVi09} for $N=2$ (in this last work, the focusing case $\gamma=1$ below the mass of the ground-state is also treated). The proofs are mainly based on technics developed for the energy-critical NLS (see e.g. \cite{Bo98IMRN}, \cite{Bo99JA}, \cite{TaVi05}, \cite{Ta05NY} and  \cite{KeMe06}).

In all these studies, a global Strichartz norm (in the mass-critical case, the $L^{\frac{4}{N}+2}$ norm)  appears as the relevant norm to control. In this work we consider 
$$
I(\delta)=\sup_{\|f\|_{L^2(\cR^N)}=\delta}\iint_{\RR\times \RR^N}
|u(t,x)|^{\frac4N+2}\,dt \,dx,
$$
where $\delta>0$ is small and $u$ is the solution to \eqref{CP}. The results cited above imply that $I(\delta)$ is finite for small $\delta$, and, in the defocusing case with $N\geq 2$, for large $\delta$ if we restrict the maximum to radial solutions. A natural extension to Theorem \ref{T:Foschi} would be to show that this maximum is achieved by a unique solution (up to symmetries) of \eqref{CP}
and give a precise estimate of $I(\delta)$.

Our main result is the following:
\begin{theo}
 \label{T:main}
Fix $\gamma\in \{-1,+1\}$.
There exists a $\delta_0>0$ such that for all $\delta$ in
$(0,\delta_0)$, the maximum $I(\delta)$ is attained: there exists a
solution $u_{\delta}$ of \eqref{CP} with initial condition
$f_{\delta}$ such that
$$
\|f_{\delta}\|_{L^2}=\delta,\quad I(\delta)= \iint_{\RR\times\RR^N}
\left|u_{\delta}(t,x)\right|^{\frac4N+2}\,dt\,dx.
$$
If $N=1$ or $N=2$, the maximizer $u_{\delta}$ is unique up to the transformations \eqref{invariances}, \eqref{PC} of the equation. Furthermore, as $\delta\to 0$,
\begin{equation}
\label{estimate_intro}
I(\delta) = C_S\delta^{\frac 4N+2}+\gamma D_{N} \delta^{\frac
8N+2}+\OOO\left(\delta^{\frac{12}{N}+2}\right),
\end{equation} 
where
$\ds D_1 = \frac1{\pi} \sum_{k\geq 1}\frac{(2k)!}{k\, 9^{k} \,
(k!)^2}\approx 0.0867$
and $\ds D_2 = \frac1{2\pi} \,\ln \frac43 \approx 0.0458$.
\end{theo}

\begin{remark}
In particular, in the focusing case in $1D$ and $2D$, the maximum of the Strichartz
norm is, for small data,  higher than in the linear case. In the defocusing case, the effect of the nonlinearity
is to lower this maximum.
\end{remark}
\begin{remark}
The constant $D_{N}$ may be expressed as
\begin{equation}
 \label{def_c*1}
D_N = -\left(2+\frac{4}{N}\right)\im \iint |G(t)|^{\frac{4}{N}}
\overline{G}(t) \int_0^t e^{i\frac{(t-s)}{2}\Delta}
\left(|G(s)|^{\frac{4}{N}}G(s)\right)ds\,dt\,dx.
\end{equation}
\end{remark}
\begin{remark}
The proof also shows that in $1D$ and $2D$, the initial condition of any maximizer with
small mass $\delta$ is (after transformations) close to $\delta G_0$, where $G_0$ is the normalized Gaussian. See Proposition \ref{P:close_G} and Remark \ref{R:cond_phi} for a precise statement.
\end{remark}

Estimates of Strichartz norms for critical nonlinear problems are only known in a few cases. Super-exponential bounds were obtained by T.~Tao for radial defocusing energy-critical equations: Schr\"odinger equation in space dimension higher than $3$ \cite{Ta05NY}, and wave equation  in $3D$ \cite{Tao06DPDE}. An equivalent of the maximizum is given in \cite{DuMe09} for the energy-critical focusing Schr\"odinger and wave equations (in space dimensions $3$, $4$ and $5$), close to the energy threshold given by the stationary solution. 

The fact that the maximum of the Strichartz norm is attained is new for a nonlinear equation. The proof of this result is based on time-dependent adaptation
to concentration-compactness arguments (see e.g. \cite{Li85Rea}) and on a super-additivity
property of $I(\delta)$ which we show by general estimates on small
solutions of \eqref{CP}. As stated in Proposition \ref{P:large_data_maximizer}, the proof would extend to larger data provided the scattering of all solutions and the super-additivity properties are shown for those data also. This proof is flexible and should also easily
adapt to other equations, e.g. the energy-critical NLS and wave equations for small data and (together with the methods of \cite{DuMe09}) close to the energy threshold. 

On the other hand, the proof of the uniqueness of the maximizer and of the 
estimate \eqref{estimate_intro} is specific to the mass-critical problem,
and strongly relies on the results of \cite{Foschi07} and
\cite{HuZh06}. A key element is the nondegeneracy of the Gaussian for the nonlinear problem, in the orthogonal space of the null directions related to the invariances of the equation:
\begin{theo}
 \label{T:ortho}
Assume $N=1,2$.
There exists $c>0$ such that if $\varphi\in L^2$ satisfies the
following orthogonality properties ($x \in \cR^N$)
\begin{equation}
 \label{ortho_intro}
\int \varphi\, G_0=\int \varphi\, |x|^2 G_0=0, \quad \int \varphi\,
x G_0=0_{\RR^N},
\end{equation}
then
$$
Q(\varphi)\geq c\|\varphi\|^2_{L^2},
$$
where $Q$ is the quadratic form associated to the second derivative of the mapping
$$ f \mapsto 
C_S \left(\int |f|^2\,dx\right)^{1+\frac{2}{N}}-\iint
\left|e^{i\frac t2 \Delta}f\right|^{2+\frac{4}{N}}dt\,dx
$$
from $L^2$ to $[0,\infty)$, at the critical point $f=G_0$.
\end{theo}
We refer to \eqref{defQ} for an expression of $Q$. This result is an analogue, for the Strichartz estimate, to the non-degeneracy of the maximizer $\frac{1}{(1+|x|^2)^{\frac{N-2}{2}}}$  for the Sobolev imbedding $\dot{H}^1(\RR^N) \hookrightarrow L^{\frac{2N}{N-2}}(\RR^N)$ (see \cite{Re90}). 

To show Theorem \ref{T:ortho}, we apply a lens tranform (\cite{Niederer74,RyVaLiTiBu00,Carles02}), related to the pseudo-conformal inversion, to the solutions of \eqref{CP}, which turns the Laplace operator into the harmonic operator $-\Delta+|x|^2$. The result then follows from explicit computations and a formula of Wei-Min Wang \cite{Wang08} on products of
eigenfunctions for the harmonic oscillator.

The outline of the paper is as follows. In Section \ref{S:maximizer}
we show that the maximizer is attained and in Section
\ref{S:estimate} we prove the estimate on $I(\delta)$. In Section \ref{S:uniqueness} we show the uniqueness of the maximizer. Section
\ref{S:quadratic} is devoted to the proof of Theorem \ref{T:ortho}. 

\subsection*{Acknowledgment}
The authors would like to thank Keith Rogers for pointing out the
article \cite{Wang08}. 

This project was partially supported by the French ANR grant ONDNONLIN. S.R. was partially supported by the NSF grant DMS-0808081.  Part of the project was done at the \textit{Institut Henri Poincar\'e} in Paris during the special trimester \textit{Ondes non-lin\'eaires et dispersion} (april-july 2009).

\section{Existence of a maximizer}
\label{S:maximizer}

In this section, where there is no restriction on the dimension $N\geq 1$, we show the first part of Theorem \ref{T:main}:
\begin{prop}
 \label{P:maximizer}
There exists $\delta_0>0$ such that if $\delta \in (0,\delta_0)$,
then there exists a solution $u_{\delta}$ of \eqref{CP}, with
initial condition $f_{\delta}$ such that
\begin{equation}
\label{E:maximizer}
\|f_{\delta}\|_{L^2(\cR^N)} = \delta \quad\text{and} \quad
\iint_{\RR\times\cR^N} \left|u_{\delta}\right|^{\frac{4}{N}+2}\,dt\,dx=I(\delta).
\end{equation}
\end{prop}
After some preliminaries (\S \ref{SS:preliminaries}) we show in \S
\ref{SS:super_add} a crucial super-additivity property of
$I(\delta)$, which relies on rough estimates of $I(\delta)$ and its
growth rate. In \S \ref{SS:proof_existence} we use this property to
prove Proposition \ref{P:maximizer} by concentration-compactness
arguments.

\subsection{Profile decomposition}
\label{SS:preliminaries}

We recall here from \cite{MeVe98} a profile decomposition
adapted to the Strichartz estimate for the linear equation
\eqref{CPlin}. We start with a long time perturbation result for the equation
\eqref{CP}.

\begin{lemma}[Long time perturbation]
 \label{L:longperturb}
Let $A>0$. There exists $C=C(A)>0$ and a small $\delta_0=\delta_0(A)>0$
such that the following holds: Let $u\in C^0(\RR,L^2_x)$
and solves
$$
i\partial_t u + \frac 12\Delta u + \gamma |u|^{\frac 4N}u=0.
$$
Let $\tilde u = \tilde{u}(x,t) \in C^0(\RR,L^2_x)$ and define
$$
e = i \partial_t \tilde u + \frac 12\Delta \tilde u + \gamma |\tilde u|^{\frac 4N}
\tilde u.
$$
Assume
$
\|\tilde{u} \|_{L^{\frac 4N+2}_{t,x}}\leq A$, 
and for some $\eps < \delta_0$
\begin{equation*}
\|e\|_{L^{\frac{2(N+2)}{N+4}}_{t,x}} \leq \varepsilon \quad
\text{and}\quad \left\|e^{i\frac{(t-t_0)}{2}\Delta}(u(t_0)-\tilde
u(t_0))\right\|_{L^{\frac4N+2}_{t,x}} \leq \varepsilon,
\end{equation*}
then
$$
\|u - \tilde u \|_{L^{\frac4N+2}_{t,x}} \leq C \, \varepsilon.
$$
\end{lemma}
We skip the proof of Lemma \ref{L:longperturb}. We refer to \cite{Bo99JA}, \cite{TaVi05}, \cite{CoKeStTaTa06}, \cite{KeMe06} for similar result for the energy-critical case, \cite{HoRo08} for a subcritical case and \cite[Lemma 3.1]{TaViZh08} for a statement close to Lemma \ref{L:longperturb} in the mass-critical case.

We next turn to the profile decomposition.
If $\Gamma_0=\left\{\rho_0,t_0,\xi_0,x_0\right\}\in
(0,+\infty)\times \RR\times \RR^N\times \RR^N$, and $u$ is a
function of space and time, we will denote by $\Gamma_0(u)$ the
function
\begin{equation}
 \label{def_Gamma}
\Gamma_0(u)=\rho_0^{\frac{N}{2}} e^{ix \cdot \xi_0}
e^{-i\frac{t}{2} \left|\xi_0\right|^2} u \left(\rho_0^2 t + t_0,
\rho_0 \left(x-\frac{t}{2}\xi_0\right)+x_0\right).
\end{equation}
As we have seen in the introduction, if $u$ is a solution to the
linear equation \eqref{CPlin} (respectively, to the nonlinear
equation \eqref{CP}), then $\Gamma_0(u)$ is also a solution to
\eqref{CPlin} (respectively, to \eqref{CP}). We say that two
sequences of transformations $\left\{\Gamma_n^1\right\}_n$ and
$\left\{\Gamma_n^2\right\}_n$ are orthogonal when
\begin{equation}
 \label{def_ortho}
\lim_{n\to \infty} \frac{\rho_n^1}{\rho_n^2} +
\frac{\rho_n^2}{\rho_n^1} + \frac{\left|\xi_n^1 -
\xi_n^2\right|}{\rho_n^1} + \left|t_n^1-t_n^2\right| +
\left|\frac{t_n^1}{2}\, \frac{\xi_n^1-\xi_n^2}{\rho_n^1}
+x_n^1-x_n^2\right|=+\infty.
\end{equation}
We recall from \cite[Theorem 2]{MeVe98} (see \cite{Keraani06} in space dimension $1$, \cite{BeVa07} for general space dimension) , the following
profile decomposition result:
\begin{lemma}
\label{L:cc_MV}
Let $\{f_n\}$ be a bounded sequence in $L^2(\RR^N)$. Then
there exists a subsequence of $\{f_{n}\}$ (still denoted by
$\{f_{n}\}$), a family $\{U^j\}_{j\geq 1}$ of solutions to
\eqref{CPlin}, and sequences of parameters $\{\Gamma_n^j\}_n$, such
that if $j\neq k$, $\left\{\Gamma_n^j\right\}_n$ is orthogonal to
$\left\{\Gamma_n^k\right\}_n$ and for all $J$,
\begin{equation}
 \label{profile}
f_{n}(x)=\sum_{j=1}^J \Gamma_n^j\left(U^j\right)(0,x)+h_{n}^J(x) ,
\end{equation}
where
$$
\lim_{J\to +\infty} \limsup_{n\to \infty} \left\| e^{i\frac{t}{2}
\Delta} h_{n}^J \right\|_{L^{\frac4N +2}_{t,x}}=0.
$$
\end{lemma}
\begin{remark}
 \label{R:pythagore}
As a consequence of the orthogonality of the transformations
$\Gamma_n^j$, the following Pythagorean expansions hold for all
$J\geq 1$:
\begin{align}
 \label{pythagore_L2}
\left\|f_{n}\right\|^2_{L^2}-\sum_{j=1}^J \left\|U^j(0)\right\|_{L^2}^2-\left\| h_{n}^J\right\|^2_{L^2}&\underset{n\to +\infty}{\longrightarrow}0,\\
\label{pythagore_L4}
\left\|e^{i\frac{t}{2}\Delta}f_{n}\right\|_{L^{\frac{4}{N}+2}_{t,x}}^{\frac{4}{N}+2}
-\sum_{j=1}^J \left\|U^j\right\|_{L^{\frac4N +2}_{t,x}}^{\frac4N
+2} - \left\|e^{i\frac{t}{2}\Delta} h_{n}^J \right\|_{L^{\frac4N
+2}_{t,x}}^{\frac4N +2}&\underset{n\to +\infty}{\longrightarrow}0.
\end{align}
\end{remark}

Let $\{f_n\}_n$ be a sequence in $L^2$ and assume that the corresponding solution to \eqref{CP} is globally defined and satisfies $\|f_{n}\|_{L^{\frac 4N+2}_{t,x}} <\infty$. Consider the profile
decomposition given by Lemma \ref{L:cc_MV}. Let $V^j$ be the
nonlinear profile associated to $\{U^j,t_n^j\}_n$, that is the
unique solution of \eqref{CP} such that
$$
\lim_{n\to \infty} \left\|U^{j} \left(t_n^j\right)- V^j \left(t_n^j
\right) \right\|_{L^2}=0.
$$
Assume also that the $V^j$'s are globally defined and such that $\|V^j\|_{L^{2+\frac{4}{N}}}$ is finite for all $j$.
Combining 
Lemmas \ref{L:longperturb} and \ref{L:cc_MV}, one gets a nonlinear version of the decomposition \eqref{profile}:
\begin{corol}
 \label{C:nonlinearPD}
Let $\{f_n\}_n$ is as above and $\{u_n\}_n$ be the sequence of
solutions to \eqref{CP} with initial conditions $\{f_n\}_n$. Then
\begin{equation}
 \label{NLprofile}
u_n(t,x)
= \sum_{j=1}^J \Gamma_n^j\left(V^j\right)(t,x) + h_n^J(t,x)+r_n^J(t,x)
\end{equation}
with
$$
\lim_{J\to +\infty} \lim_{n\to +\infty} \left(\left\| r_{n}^J \right\|_{L^{\frac4N +
2}_{t,x}} +\sup_{t\in \RR} \left\|
r_n^J(t) \right\|_{L^2}\right)=0.
$$
\end{corol}

\begin{remark}
 \label{R:NL_pythagore}
Using the orthogonality of the sequences of transformations $\{\Gamma_n^j\}_n$, it is easy to check that 
\begin{align}
 \label{NLS_L4}
\lim_{J\to \infty}\limsup_{n\to \infty}\left|
\left\| u_{n} \right\|_{L^{\frac{4}{N}+2}_{t,x}}^{\frac{4}{N}+2}
-\sum_{j=1}^J\left\|V^j\right\|_{L^{\frac4N +2}_{t,x}}^{\frac4N +2}\right|=0.
\end{align}
\end{remark}

\subsection{A superadditivity property of the maximum}
\label{SS:super_add}

In this paragraph we give various estimates on $I(\delta)$. The main
result is the following proposition, which is one of the steps
(along with a concentration-compactness argument) in showing that the
maximizer is attained:
\begin{prop}
\label{P:super_add}
 There exists $\delta_0>0$ such that if $0<\sqrt{\alpha^2+\beta^2}<\delta_0$, then
\begin{equation*}
I(\alpha)+I(\beta)<I\left(\sqrt{\alpha^2+\beta^2}\right).
\end{equation*}
\end{prop}
\begin{remark}
Superadditivity (or subadditivity for minimizers) conditions
are classical in this context (see \cite[Subsection I.2]{Lions84a}).
\end{remark}

The proof of Proposition \ref{P:super_add} relies on two estimates
on $I(\delta)$ that we treat in Lemmas \ref{L:first_est} and \ref{L:growthrate} below.
\begin{lemma}
 \label{L:first_est}
There exists a constant $C_0>0$ such that for small $\delta>0$,
\begin{equation}
  \label{E:1st}
\left|I(\delta)-C_S\delta^{\frac4N+2} \right| \leq C_0\delta^{\frac{8}{N}+2},
\end{equation}
where $C_S$ is the best constant for the Strichartz inequality 
\begin{equation}
\label{strichartz2}
\iint \left|e^{i\frac{t}{2}\Delta}f\right|^{\frac{4}{N}+2}dt\,dx\leq C_S \|f\|_{L^{2}}^{\frac 4N+2}. 
\end{equation} 
\end{lemma}
Before proving this lemma, we start by a straightforward consequence of the small data
well-posedness theory for equation \eqref{CP} (see \cite{CaWe90}).
\begin{claim}
\label{C:first_est}
There exists a constant $C>0$ such that if $\|f\|_{L^2}$ is small, then
\begin{equation*}
 \left\|e^{i\frac{t}{2}\Delta}f - u\right\|_{L^{\frac4N +2}_{t,x}} \leq C\|f\|^{\frac{4}{N}+1}_{L^2},
\end{equation*}
where $u$ is the solution of \eqref{CP} with initial condition $f$.
\end{claim}
\begin{proof}[Sketch of proof]
The Cauchy problem theory for \eqref{CP} implies that for small initial data
$$
\|u\|_{L^{\frac4N +2}_{t,x}}\leq 2\|f\|_{L^2}.
$$
Since
$$
u(t) = e^{i\frac{t}{2} \Delta} f + i \gamma \int_0^{t}
e^{\frac{i}{2}(t-s) \Delta}|u(s)|^{\frac{4}{N}} u(s)ds,
$$
the claim follows from Theorem \ref{T:Foschi} and the
Strichartz estimate
$$
\left\|\int_{0}^t e^{\frac{i}{2}(t-s) \Delta}\varphi(s)ds
\right\|_{L^{\frac4N +2}_{t,x}} \leq C\|\varphi\|_{L^{\frac{2(N+2)}{N+4}}_{t,x}}.
$$
\end{proof}

\begin{proof}[Proof of Lemma \ref{L:first_est}]
Let $u$ be a solution of \eqref{CP} with initial condition $f$
such that $\|f\|_{L^2(\cR^N)}=\delta$. Then
\begin{multline*}
\left| \iint \left| e^{i\frac{t}{2} \Delta} f(x)\right|^{\frac4N+2}\,dt\,dx
- \iint \left| u(t,x) \right|^{\frac4N+2}\,dt\,dx \right| \\
\leq C \left| \,  \left\| e^{i\frac{t}{2}\Delta}f\right\|_{L^{\frac4N+2}_{t,x}}
- \left\| u \right\|_{L^{\frac4N+2}_{t,x}} \right|
\left(\left\|e^{i\frac{t}{2}\Delta} f\right\|_{L^{\frac4N+2}_{t,x}}^{\frac{4}{N}+1}
+\left\|u\right\|_{L^{\frac4N+2}_{t,x}}^{\frac{4}{N}+1} \right) \\
\leq C \, \left\| e^{i\frac{t}{2}\Delta} f- u \right\|_{L^{\frac4N+2}_{t,x}} \, \delta^{\frac{4}{N}+1}
\leq C \delta^{\frac{8}{N}+2},
\end{multline*}
where the last line follows from the triangle inequality and then
from Claim \ref{C:first_est}. Applying the previous inequality to
the initial data $f= \delta F_0$, where $F_0$ is the initial condition of a maximizer for Strichartz estimate \eqref{strichartz2}, and then to a sequence
$\left\{f_{n}\right\}_n$ such that $\|f_{n}\|_{L^2}=\delta$ and $\iint
|u_n|^{\frac{4}{N}+2}\to I(\delta)$, we  obtain \eqref{E:1st}.
\end{proof}

We next estimate the rate of growth of $I(\delta)$.
\begin{lemma}
 \label{L:growthrate}
If $\delta$ is small and $\eps \leq \frac12 \delta$, then
\begin{equation}
 \label{E:rate}
I(\delta) + c_1 \,\delta^{\frac4N +1} \eps\leq 
I(\delta+\eps) \leq I(\delta) + C_1 \, \delta^{\frac4N +1} \,\eps,
\end{equation}
where $c_1 = \frac4N \, C_S $ and $C_1 = 2\left(\frac4N+2\right) \, C_S$.
\end{lemma}
\begin{proof}
\noindent \emph{Step 1.}
We first show that there exist 
$C_2,\epsilon_0>0$ such that if $f\in L^2$ with
$\|f\|_{L^2}+\epsilon \leq \epsilon_0$, $u$ is the solution of
\eqref{CP} with the initial condition $f$, and $v_{\epsilon}$ is the
solution of \eqref{CP} with the initial condition $(1+\epsilon)f$,
then
$$
\left|(1+\epsilon)^{\frac4N+2} \iint \big|u\big|^{\frac4N+2}
- \iint \big|v_{\epsilon}\big|^{\frac4N+2} \right|
\leq C_2 \, \epsilon \, \|f\|_{L^2}^{\frac{8}{N}+2}.
$$
First, observe that $u_{\epsilon}=(1+\epsilon)u$ is a solution to
the equation
$$
i\partial_t u_{\epsilon}+\frac{1}{2}\Delta u_{\epsilon} +
\frac{1}{(1 + \epsilon)^{\frac 4N}} |u_{\epsilon}|^{\frac{4}{N}}
u_{\epsilon}=0, \quad u_{\epsilon \restriction t=0}=(1+\epsilon)f.
$$
We rewrite the above equation as
$$
i\partial_t u_{\epsilon}+\frac{1}{2}\Delta
u_{\epsilon}+|u_{\epsilon}|^{\frac
4N}u_{\epsilon}=\left(1-\frac{1}{(1+\epsilon)^{\frac
4N}}\right)|u_{\epsilon}|^{\frac 4N}u_{\epsilon},
$$
noting that for small $\epsilon$,
Strichartz estimate implies
\begin{align*}
\bigg\|\Big(1 - \frac{1}{(1 + \epsilon)^{\frac 4N}} \Big)|u_{\epsilon}|^{\frac 4N} u_{\epsilon} \bigg\|_{L^{\frac{2(N+2)}{N+4}}_{t,x}} & \leq
C\, \epsilon \Big\||u_{\epsilon}|^{1+\frac{4}{N}} \Big\|_{L^{\frac{2(N+2)}{N+4}}_{t,x}}\\
= C\,\epsilon \|u_{\epsilon} \|^{1+\frac{4}{N}}_{L^{\frac4N+2}_{t,x}}
& \leq C \epsilon \|f\|^{1+\frac{4}{N}}_{L^2}.
\end{align*}
Since $v_{\epsilon}$ is a solution of
$$
i\partial_t v_{\epsilon}+\frac 12\Delta v_{\epsilon}+|v_{\epsilon}|^{\frac 4N} v_{\epsilon}=0,
\quad v_{\epsilon\restriction t=0}=(1+\epsilon)f,
$$
by the long time perturbation Lemma \ref{L:longperturb}, we get
$$
\left\|u_{\epsilon} - v_{\epsilon}\right\|_{L^{\frac4N+2}_{t,x}}
\leq C \epsilon\|f\|^{\frac{4}{N}+1}_{L^2}.
$$
Hence,
\begin{align*}
\left|\iint \left|u_{\epsilon}\right|^{\frac4N+2}\,dt\,dx
- \iint \left|v_{\epsilon}\right|^{\frac4N+2} \,dt\,dx \right|
&\leq C \left\|u_{\epsilon}-v_{\epsilon}\right\|_{L^{\frac4N+2}_{t,x}} \,
\|f\|^{\frac{4}{N}+1}_{L^2}\\
&\leq C \epsilon \|f\|^{\frac{8}{N}+2}_{L^2},
\end{align*}
which concludes Step 1.

\smallskip

\noindent\emph{Step 2.}
Let $\eps$, $\delta >0$. First, we show the lower bound of $I(\delta+\eps)$.
Let $f \in L^2(\cR^N)$ be such that
\begin{equation}
\label{cond_u_delta} \|f\|_{L^2}=\delta \quad \text{and} \quad \iint
|u(t,x)|^{\frac4N +2} \,dt \,dx \geq I(\delta)-\delta^{\frac8N +1} \eps,
\end{equation}
where $u$ is the corresponding solution of \eqref{CP} and we used
the supremum property of $I(\delta)$. Let $u_{\eps}$ be the solution
of \eqref{CP} with the initial condition
$\left(1+\frac{\eps}{\delta}\right)f$. Then
$\|u_{\eps}(0)\|_{L^2}=\delta+\eps$. By Step 1,
\begin{align*}
I(\delta+\eps) &\geq \iint |u_{\eps}(t,x)|^{\frac4N +2} \,dt\,dx \\
& \geq \left(1+\frac{\eps}{\delta}\right)^{\frac4N +2}
\iint|u(t,x)|^{\frac4N +2} - C_2 \frac{\eps}{\delta}\, \delta^{\frac 8N+2}.
\end{align*}
By \eqref{cond_u_delta}, we get
\begin{equation*}
I(\delta +\eps) \geq \left[1+\left(\frac4N +2\right)\, \frac{\eps}{\delta} \right]
\left(I(\delta)-\delta^{\frac8N +1}\eps\right) - C_2 \delta^{\frac 8N+1} \eps.
\end{equation*}
Lemma \ref{L:first_est} implies $I(\delta) \geq C_S
\delta^{\frac4N+2} - C_0 \delta^{\frac8N+2}$, hence,
\begin{align*}
I(\delta+\eps) & \geq I(\delta)+C_S \,\left(\frac4N +2\right) \,\delta^{\frac4N+1}\eps \\
& - \left[ \left(\frac4N +2\right) C_0 + \left(1+\left(\frac4N+2\right) \frac{\eps}{\delta} \right)
+ C_2 \right] \delta^{\frac8N +1}\eps.
\end{align*}
Now if $\eps < \frac12 \delta$ and 
$$\delta < \left(
\frac{C_S}{4+6C_0+C_2} \right)^{N/4},$$ 
the last term in the
expression above will be less than $2C_S \delta^{\frac4N+1}\eps $,
and thus, the right side in \eqref{E:rate} follows with $c_1 =
\frac4N \, C_S$.

The upper bound on $I(\delta+\eps)$ follows similarly from Step 1
and Lemma \ref{L:first_est}, obtaining the left side in
\eqref{E:rate} with $C_1 = 2 C_S \left(\frac4N +2\right)$.
\end{proof}

We next prove Proposition \ref{P:super_add}.
\begin{proof}
Without loss of generality, we can assume $0<\alpha\leq \beta$.

\smallskip

\noindent\emph{Step 1}. We first show that there exists a large
constant $C_3>0$ such that the conclusion of the proposition holds
if
\begin{equation}
\label{large_alpha}
 C_3\beta^{\frac{2}{N}+1}\leq \alpha\leq \beta.
\end{equation}
By Lemma \ref{L:first_est}, 
\begin{align*}
I(\alpha)+I(\beta)&\leq C_S\alpha^{\frac{4}{N}+2}+C_S\beta^{\frac{4}{N}+2}+2C_0\beta^{\frac 8N+2},\\
\text{and}\quad
 C_S\left(\alpha^2+\beta^2\right)^{\frac{2}{N}+1}&\leq I\left(\sqrt{\alpha^2+\beta^2}\right)+2C_0\beta^{\frac 8N+2}.
\end{align*}
There is a constant $\kappa_N>0$ such that $1+x^{\frac{2}{N}+1}+\kappa_N x\leq (1+x)^{\frac{2}{N}+1}$ for $x\in [0,1]$.
As a consequence, $\alpha^{\frac 4N+2}+\beta^{\frac 4N+2}+\kappa_N\,\beta^{\frac 4N}\alpha^2\leq \left(\alpha^2+\beta^2\right)^{\frac 2N+1}$. Combining with the previous estimates, we get
$$
I(\alpha)+I(\beta)+C_S\,\kappa_N\,\beta^{\frac
4N}\alpha^2-4C_0\beta^{\frac 8N+2}
\leq I \left(\sqrt{\alpha^2+\beta^2}\right),
$$
which yields the announced result if $C_3$ is chosen large in
\eqref{large_alpha}.

\smallskip

\noindent\emph{Step 2}. We next show that the conclusion of the
Proposition still holds if
\begin{equation}
\label{small_alpha}
 0<\alpha<C_3\beta^{\frac{2}{N}+1},
\end{equation}
where $C_3$ is the constant defined in Step 1. Choosing $\delta_0$
small enough, $\beta\leq \delta_0$ and  \eqref{small_alpha} imply
$$
\frac{\alpha^2}{4\beta} \leq \sqrt{\alpha^2+\beta^2}-\beta\leq
\frac{\beta}{2}.
$$
By Lemma \ref{L:growthrate}, with $\delta=\beta$ and
$\eps=\sqrt{\alpha^2+\beta^2}-\beta$,
$$
I(\beta) \leq I\left(\sqrt{\alpha^2+\beta^2}\right)-c_1\beta^{\frac
4N+1}\left(\sqrt{\alpha^2+\beta^2}-\beta\right)\leq
I\left(\sqrt{\alpha^2+\beta^2}\right)-\frac{c_1}{4}\beta^{\frac
4N}\alpha^2.
$$
Combining with Lemma \ref{L:first_est} we get, taking a smaller $\delta_0$ if necessary,
\begin{multline*}
I(\alpha) + I(\beta)\leq I \left(\sqrt{\alpha^2 + \beta^2}\right) -
\frac{c_1}{4}\beta^{\frac 4N}\alpha^2 + 2C_S\alpha^{\frac{4}{N}+2}
\\ \leq I\left(\sqrt{\alpha^2+\beta^2}\right)+\alpha^2\beta^{\frac 4N}
\left(2C_SC_3^{\frac 4N}\beta^{\frac{8}{N^2}}-\frac{c_1}{4}\right),
\end{multline*}
which shows that the conclusion of the proposition holds also in
this case, provided $\delta_0>0$ is small enough.
\end{proof}

\subsection{Proof of the existence of the maximizer}
\label{SS:proof_existence}
Let us show Proposition \ref{P:maximizer}. We will prove the following more general result:
\begin{prop}
\label{P:large_data_maximizer}
Assume that there exists a constant $A>0$ such that 
\begin{enumerate}
\item \label{scattering} Scattering: for all $f\in L^2$ such that $\|f\|_{L^2}\leq A$, the solution $u$ of \eqref{CP} with initial condition $f$ is globally defined and
$$ \delta \leq A\Longrightarrow I(\delta)<\infty.$$
\item \label{superA} Superadditivity:
if $0<\sqrt{\alpha^2+\beta^2}=A$, and $\alpha,\beta>0$, then
\begin{equation*}
I(\alpha)+I(\beta)<I\left(A\right).
\end{equation*}
\end{enumerate}
Then there exists a solution $u_A$ of \eqref{CP} with initial condition $f_A\in L^2$ such that
$$ \|f_A\|_{L^2}=A,\quad \iint |u_A|^{2+\frac 4N}=I(A).$$
\end{prop}
In view of the small data global well-posedness theory and Proposition \ref{P:super_add}, Proposition \ref{P:large_data_maximizer} implies Proposition \ref{P:maximizer}. Let us prove Proposition \ref{P:large_data_maximizer}.

Let $\{u_n\}_n$ be a sequence of
solutions to \eqref{CP} with initial data $f_{n}$ such that
$$
\left\|f_{n}\right\|_{L^2(\cR^N)} = A,
\quad \lim_{n\to \infty} \iint_{\cR^N} |u_n|^{\frac{4}{N}+2} = I(A).
$$
We will show that there exist a subsequence of $\{u_n\}_n$ and a
sequence $\{\Gamma_n\}_n$ of transformations such that
$\left\{\Gamma_n(u_n)\right\}_n$ converges strongly in $L^2$. Consider, after extraction, a profile decomposition
of the sequence $\{f_{n}\}_n$:
\begin{equation}
 \label{profile_decompo}
f_{n}=\sum_{j=1}^J \Gamma_n^j\left(
U^j\right)_{\restriction t=0}+h_{n}^J.
\end{equation}
It is sufficient to show that $U^j=0$ except for one $j$ and that $\lim_{n\to
\infty} \|h_{n}^J\|_{L^2}=0$, which we will do in two steps.

\smallskip

\noindent\emph{Step 1: no dichotomy}. 
First assume that there are at least two nonzero profiles, say
$U^1\neq 0$ and $U^2\neq 0$. Let $V^1$ be the nonlinear profiles associated to $\{U^1,t_n^1\}$ and $V_n$ the solution of \eqref{CP} given by
$$
V_{n}=\Gamma_n^1(V^1).
$$
Let $W_n$ be the sequence of solutions to \eqref{CP} with initial
condition
$$
W_{n}(0)=f_{n}-V_{n}(0).
$$
Let $r_n=u_n-V_n-W_n$. By assumption \eqref{scattering}, all the nonlinear profiles $V^j$ scatter. Thus, one can use Corollary \ref{C:nonlinearPD}, showing
$$
\lim_{n\to \infty} \sup_{t\in \RR} \|r_n(t)\|_{L^2}=0.
$$
Furthermore, (see \eqref{pythagore_L2} and Remark \ref{R:NL_pythagore})
\begin{gather}
 \label{ortho_L2}
\int |f_{n}|^2=\int |V_{n}(0)|^2+\int |W_{n}(0)|^2+o_n(1)\\
\label{ortho_L4} \iint |u_n|^{\frac{4}{N}+2}=\iint
|V_n|^{\frac{4}{N}+2}+\iint |W_n|^{\frac{4}{N}+2}+o_n(1).
\end{gather}
Let $\eps=\|U^1(0)\|_{L^2}$. Then for all $n$,
$
\eps=\|V_{n}(0)\|_{L^2}$.
By \eqref{ortho_L2},
$$
\|W_{n}(0)\|_{L^2}^2=A^2-\eps^2+o_n(1).
$$
By our assumptions, $\eps>0$ (otherwise, $U^1$ would be zero) and
$A^2-\eps^2>0$ (otherwise, $U^2$ would be zero). Using that
$\iint |u_n|^{\frac{4}{N}+2}$ tends to $I(A)$ as $n\to \infty$, and that by Lemma \ref{L:longperturb}, $\limsup_n \int |W_n|^{\frac{4}{N}+2}\leq I(\sqrt{A^2-\eps^2})$, we get by \eqref{ortho_L4}
\begin{equation}
 \label{absurd_I}
I(A)\leq I(\eps)+I\left(\sqrt{A^2-\eps^2}\right).
\end{equation}
This contradicts assumption \eqref{superA}, concluding Step 1.

\smallskip

\noindent\emph{Step 2: non vanishing and the end of the proof.}
There must be one nonzero profile in
\eqref{profile_decompo}. If not, then
$$
\lim_{n\to \infty} \iint |u_n|^{\frac{4}{N}+2}=0,
$$
showing that $I(A)=0$, a contradiction. It remains to show that
the remainder $h_n=h_{n}^J$ in \eqref{profile_decompo} tends to $0$ in
$L^2$. Denote by
$$
\eps=\lim_{n\to \infty}\|h_{n}\|_{L^2},$$
then, using again Lemma \ref{L:longperturb}, we get $I(A)\leq I(\sqrt{A^2-\eps^2})$, which shows by assumption \eqref{superA} that
$\eps=0$.

Denoting by $U^1$ the only nonzero profile in \eqref{profile_decompo},
we have shown that $(\Gamma_n^1)^{-1}(u_n)$ tends to $U^1$ in
$L^2$, and therefore,
$$
\|U^1\|_{L^2}=A,\quad \iint|U^1|^{\frac{4}{N}+2}=I(A),
$$
concluding the proof of the proposition.
\qed

\section{Estimate of the maximum of the Strichartz norm}
\label{S:estimate}

In the remainder of the paper, we restrict ourselves to $1D$ and $2D$. In this section we prove the second part of Theorem \ref{T:main}:
\begin{prop}
\label{P:estimate} Assume that $N=1$ or $N=2$. Then as $\delta\to 0$,
$$
I(\delta)=\iint\left|u_{\delta}\right|^{2+\frac
4N}=C_S\delta^{2+\frac 4N}+\gamma D_N\delta^{2+\frac
8N}+\OOO\left(\delta^{2+\frac{12}{N}}\right),
$$
where
$\ds D_1 = \frac1{\pi} \sum_{k\geq 1}\frac{(2k)!}{k\, 9^{k} \,
(k!)^2} \approx 0.0867$ and $\ds D_2 = \frac1{2\pi} \,\ln \frac43
\approx 0.0458$.
\end{prop}

Before proving Proposition \ref{P:estimate}, we define the quadratic form associated to the maximum of the Strichartz estimate that appears in Theorem \ref{T:ortho}. By  Theorem \ref{T:Foschi}, if $G$ is the Gaussian
solution defined by \eqref{def_G} and $\varphi\in L^2$, then
$$
C_S \left(\int |G_0+\varphi|^2\right)^{1+\frac{2}{N}}-\iint
\left|G+e^{i\frac t2 \Delta}\varphi\right|^{2+\frac{4}{N}}\geq 0.
$$
Expanding the above inequality and using that $G$ is a maximizer, we
obtain that the linear part vanishes, i.e.,
\begin{equation}
\label{linear_lagrange}
\forall \varphi\in L^2,\quad C_S\re \int G_0\varphi=\re \iint |G|^{\frac 4N}\overline{G}\,e^{i\frac{t}{2}\Delta}\varphi.
\end{equation} 
The expansion at second order in $\varphi$ yields
\begin{equation}
 \label{devt_Q}
C_S \left(\int |G_0+\varphi|^2\right)^{1+\frac{2}{N}}-\iint
\left|G+e^{i\frac t2 \Delta}\varphi\right|^{2 + \frac{4}{N}} =
Q(\varphi)+\OOO\left(\|\varphi\|_{L^2}^3\right),
\end{equation}
where $Q$ is a (real) nonnegative symmetric quadratic form on $L^2$
defined by
\begin{multline}
 \label{defQ}
Q(\varphi) = C_S\left[\frac{N+2}{N}\int |\varphi|^2+\frac{4(N+2)}{N^2}
\left(\re \int G_0\varphi\right)^2\right]\\
-\frac{(N+2)^2}{N^2}\iint |G|^{\frac{4}{N}}\left|e^{i\frac t2\Delta}
\varphi\right|^2-\frac{2(N+2)}{N^2}\re \iint |G|^{\frac 4N-2}\,
\overline{G}^2\left(e^{i\frac t2\Delta}\varphi\right)^2.
\end{multline}

By the transformations of the linear equation (respectively,
multiplication by a real number, phase shift, space translation,
Galilean invariance, scaling and time translation), we have
\begin{equation}
 \label{E:orthodir1}
Q(G_0)=Q(iG_0)=Q(xG_0)=Q(ixG_0)=Q(x^2G_0)=Q(ix^2G_0)=0,
\end{equation}
if $N=1$ and
\begin{equation}
 \label{E:orthodir2}
Q(G_0)=Q(iG_0)=Q(x_jG_0)=Q(ix_jG_0)=Q(|x|^2G_0)=Q(i|x|^2G_0)=0,
\end{equation}
(where $j=1,2$) if $N=2$. Theorem \ref{T:ortho}, which will be proved in Section \ref{S:quadratic}
 states that $Q$ is positive definite in the subspace of functions in
$L^2$ that are orthogonal to the directions in \eqref{E:orthodir1} or \eqref{E:orthodir2}. This non-degeneracy property is crucial in the proof of Proposition \ref{P:estimate}, which is divided in two parts.

\subsection{Choice of the maximizer}
We first give a corollary to the linear profile decomposition that will be needed in the proof. Recall from
\eqref{def_G} the definition of the normalized Gaussian $G$ .

\begin{lemma}
 \label{L:cc}
Let $\{f_{n}\}_n$ be a sequence in $L^2(\RR^N)$ such that
\begin{equation}
\label{CV_L2_norm} \lim_{n\to \infty} \|f_{n}\|_{L^2}=1,
\end{equation}
and
\begin{equation}
 \label{CV_L4}
\lim_{n\to \infty} \iint\left| e^{i\frac{t}{2} \Delta} f_{n}
\right|^{\frac4N+2} \,dt\,dx = C_S.
\end{equation}
Then there exist a subsequence of $\left\{f_{n}\right\}_n$ (still
denoted by $\left\{f_{n}\right\}_n$), a phase $\theta_0$ and a
sequence $\{\Gamma_n\}_n$ of transformations of the form
\eqref{def_Gamma} such that
\begin{equation}
\label{to_gaussian} \lim_{n\to\infty}
\left\|f_{n}-e^{i\theta_0}\Gamma_n(G)\right\|_{L^2}=0,
\end{equation}
where $G$ is the normalized Gaussian solution defined in
\eqref{def_G}.
\end{lemma}

\begin{proof}
This is an application of Lemma \ref{L:cc_MV} and the
uniqueness result of Foschi \cite{Foschi07}.

After extraction of a subsequence, the sequence $\{f_{n}\}_n$ admits
a profile decomposition of the form \eqref{profile}. At least one of
the profiles is nonzero. Indeed, if it was not the case,
$\left\|e^{i\frac t2 \Delta}f_n\right\|_{L^{\frac4N+2}}$ would tend
to $0$, a contradiction with \eqref{CV_L4}. Reordering the profiles,
we may assume that $U^1 \neq 0$. By the Pythagorean expansion
\eqref{pythagore_L4} and by \eqref{CV_L4}
$$
C_S + o_n(1) = \iint \left|e^{i\frac t2 \Delta}f_n\right|^{\frac
4N+2}\,dt\,dx \leq C_S \left(\left\|U^1\right\|_{L^2}^{\frac4N+2} +
\left\|w_n^1\right\|_{L^2}^{\frac4N+2} \right) + o_n(1).
$$
Using that by \eqref{pythagore_L2},
$\left\|w_n^1\right\|_{L^2}^2=1-\left\|U^1\right\|_{L^2}^2+o_n(1)$,
we obtain from the previous expression that
$$
1 \leq \left(\left\|U^1\right\|_{L^2}^2\right)^{\frac4N+2} +
\left(1-\left\|U^1\right\|_{L^2}^2\right)^{\frac4N+2},
$$
which shows that $\|U^1\|_{L^2}=1$ (we already excluded the case
$\|U^1\|_{L^2}=0$), and by \eqref{pythagore_L2} again
$$
\lim_{n\to \infty} \|f_{n}- \Gamma^1_n(U^1)(0)\|_{L^2}=0.
$$
By our assumptions on $f_{n}$ we obtain, passing to the limit, that
$\|U^1(0)\|_{L^2}=1$ and $\|U^1\|_{L^{\frac4N+2}}^{\frac{4}{N}+2
} =C_S$, which shows by Theorem \ref{T:Foschi} that $U^1(0)=G_0$ up
to the symmetries of the equation (i.e., the transformations of the
form \eqref{def_Gamma} and the multiplication by a phase
$e^{i\theta_0}$), which completes the proof.
\end{proof}

\begin{prop}
 \label{P:close_G}
There exists $\delta_0>0$ such that if 
 $\left\{u_{\delta}^*\right\}_{0<\delta<\delta_0}$ is a family of maximizers, i.e. $u_{\delta}^*$ satisfies \eqref{E:maximizer}, then for all $\delta\in(0,\delta_0)$ there exists a transformation $u_\delta$ of $u_{\delta}^*$ such that $f_{\delta}=u_{\delta}(0,x)$ satisfies:
$$
f_{\delta} = \alpha_{\delta} G_0+\varphi_{\delta},\quad
\lim_{\delta\to 0^+}\frac{\alpha_{\delta}}{\delta}=1,
$$
with $\varphi_{\delta}$ satisfying the orthogonality properties
\eqref{ortho_intro} and
\begin{equation}
\label{cond_phi}
\forall \delta\in (0,\delta_0),\quad \|\varphi_{\delta}\|_{L^2}\leq C\delta^{1+\frac 2N}.
\end{equation}
\end{prop}
By ``transformation'' we mean a symmetry of \eqref{CP} which is a combination of transformations of the form \eqref{invariances} and \eqref{PC}.
\begin{remark}
 \label{R:cond_phi}
We will later improve the estimates on $\varphi_{\delta}$ and $\alpha_{\delta}$ and obtain
(see \eqref{compare_alpha_delta}, \eqref{new_varphi}):
$$
\forall \delta\in(0,\delta_0),\quad\|\varphi_{\delta}\|_{L^2}\leq C\delta^{1+\frac 4N}
\text{ and }
\left|\alpha_{\delta}-\delta\right|\leq C \delta^{1+\frac 4N}.$$
\end{remark}
\begin{proof}
The proof is divided into three steps.

\smallskip

\noindent\emph{Step 1. Closeness to $G_0$.} 
In this step we show
that if $\delta$ is small enough, there exists a transformation 
$v_{\delta}$ of $u_{\delta}^*$ which satisfies the
maximizer equations \eqref{E:maximizer} and
\begin{equation}
 \label{v0n_to_G}
\lim_{\delta\to 0} \delta^{-1}\|g_{\delta}-\delta G_0\|_{L^2}=
0,\quad \text{where}\quad g_{\delta}(x)=v_{\delta}(0,x).
\end{equation}
Arguing by contradiction, we see that it is sufficient to show that for any sequence $\delta_n\to 0$ there exists (after extraction of a subsequence) a sequence of solutions $\left\{v_{\delta_n}\right\}_n$ that are obtained as transformations of $u_{\delta_n}^*$ and satisfy \eqref{v0n_to_G}. 

By Claim \ref{C:first_est} and Lemma \ref{L:first_est}, there exists a constant $C>0$ such that
\begin{equation*}
\left|\iint |e^{i\frac t2\Delta} f_{\delta_n}^*|^{2+\frac{4}{N}}\,dt\,dx-C_S
\delta_n^{2+\frac{4}{N}}\right|\leq C\delta_n^{2+\frac{8}{N}}.
\end{equation*}
By Lemma \ref{L:cc}, we obtain after extraction of subsequences that there exist $\theta_0\in
\RR$ and a sequence of transformations $\{\Gamma_n\}$ such that
\begin{equation}
\label{tildeu_cc}
\lim_{n\to\infty} \delta_n^{-1}\left\|f^*_{\delta_n}-\delta_n
e^{i\theta_0}\Gamma_n(G)_{\restriction t=0}\right\|_{L^2}=0.
\end{equation}
Note that, by \eqref{def_G},
\begin{equation*}
\Gamma_n(G)_{\restriction
t=0}=\rho_n^{\frac{N}{2}}e^{ix\cdot\xi_n}G\left(t_n,\rho_nx+x_n\right)
=\frac{\rho_n^{\frac N2}e^{ix\cdot\xi_n}}{\pi^{N/4}(1+it_n)^{N/2}}e^{-\frac{|\rho_nx+x_n|^2}{2(1+it_n)}}.
\end{equation*}
And thus, by the change of variable $y=\frac{\rho_nx+x_n}{\sqrt{1+t_n^2}}$,
\begin{multline*}
\left\|f_{\delta_n}^*(x)-\delta_ne^{i\theta_0}\Gamma_n\left(G\right)_{\restriction t=0}\right\|_{L^2}^2=
\\ \int \left|e^{i\tau_n+i\frac{t_n|y|^2}{2}+i\frac{\sqrt{1+t_n^2}y-x_n}{\rho_n}\cdot\xi_n}
\frac{\left(1+t_n^2\right)^{\frac N4}}{\rho_n^{\frac{N}{2}}}\overline{f}_{\delta_n}^*\left(\frac{\sqrt{1+t_n^2}y-x_n}{\rho_n}\right)-\frac{\delta_ne^{-\frac{|y|^2}{2}}} {\pi^\frac{N}{4}}
\right|^2dy,
\end{multline*}
where $e^{i\tau_n}=\left(\frac{\sqrt{1+t_n^2}}{1+it_n}\right)^{\frac N2}$.
Consider the solution $w_{\delta_n}$ of \eqref{CP} with initial condition
$$ h_{\delta_n}(x)=e^{i\tau_n+i\frac{\sqrt{1+t_n^2}y-x_n}{\rho_n}\cdot\xi_n}
\frac{\left(1+t_n^2\right)^{\frac N4}}{\rho_n^{\frac{N}{2}}}\overline{f}_{\delta_n}^*\left(\frac{\sqrt{1+t_n^2}y-x_n}{\rho_n}\right),$$
and the solution $v_{\delta_n}$ of \eqref{CP} with initial condition $g_{\delta_n}=e^{i\frac{t_n|y|^2}{2}}h_{\delta_n}$.
Then $w_{\delta_n}$ is an image of $u_{\delta_n}^*$ by phase, scaling, space translation and Galilean transformation (see \eqref{invariances}). Furthermore,  $v_{\delta_n}$ is obtained from $w_{\delta_n}$ with a combination of pseudo-conformal transformation and time translation. Namely:
$$v_{\delta_n}(t,x)=
\frac{t_n^{N/2}}{(t_n^2t+t_n)^{N/2}}
e^{\frac{it_n|x|^2}{2(t_n t+1)}} w_{\delta_n} \left(\frac{t}{1 +
t_n t}, \frac{t_n x}{t_n^2 t +t_n}\right).
$$
All these transformations preserve the $L^2$ norm and the global space-time $L^{2+\frac{4}{N}}$ norm, which shows that
$$ \left\|g_{\delta_n}\right\|_{L^2}=\delta_n, \quad \iint \left|v_{\delta_n}\right|^{\frac{4}{N}+2}=I(\delta_n).$$
By \eqref{tildeu_cc},
$$ \lim_{n\to \infty} \frac{1}{\delta_n}\left\|g_{\delta_n}-\delta_nG_0\right\|_{L^2}=0,$$
concluding the first step.

\smallskip

\noindent\emph{Step 2. Orthogonality conditions.}
We next show that the statement of the proposition holds if \eqref{cond_phi} is replaced by the weaker condition
\begin{equation}
\label{cond_phi'}
\lim_{\delta\to 0}\delta^{-1}\|\varphi_{\delta}\|_{L^2}=0.
\end{equation}
For this we must show that there exists a transformation $u_{\delta}$ of $v_{\delta}$ such that $\varphi_{\delta}$ satisfies the orthogonality conditions \eqref{ortho_intro}.
Consider the unit ball
$$ B_{L^2}(G_0,1)=\left\{f\in L^2,\; \|f-G_0\|_{L^2}<1\right\},$$
and define, for small $\delta>0$, a differentiable mapping
\begin{equation*}
\Phi_{\delta}:\RR\times(0,+\infty)\times \RR\times \RR^N\times \RR^N\times B_{L^2}(G_0,1)\longrightarrow \RR\times \RR\times \RR^N\times \RR^N\times\RR
\end{equation*}
as follows. If $\theta_0\in \RR$, $\Gamma_0\in (0,+\infty)\times \RR\times \RR^N\times \RR^N$, $f\in B_{L^2}(G_0,1)$, $\tilde{u}_{\delta}$ is the solution of \eqref{CP} with initial condition $\delta f$ and
$$ U_{\delta}(x)=\delta G_0-e^{i\theta_0}\Gamma_0\left(\tilde{u}_{\delta}\right)_{\restriction t=0}=\delta G_0-e^{i\theta_0}\rho_0^{\frac{N}{2}} e^{ix\cdot\xi_0}\tilde{u}_{\delta}\left(t_0,\rho_0x+x_0\right),$$
then $\Phi_{\delta}(\theta_0,\Gamma_0,f)=(\Phi_{\delta}^1,\Phi_{\delta}^2,\Phi_{\delta}^3,\Phi_{\delta}^4,\Phi_{\delta}^5)$ is defined by
\begin{align*}
 \Phi_{\delta}^1&=\frac{1}{\delta}\im \int  U_\delta\, G_0,&
 \Phi_{\delta}^2&=\frac{1}{\delta}\re \int U_\delta\, \left(|x|^2-\frac{N}{2}\right)G_0,&
 \Phi_{\delta}^3&=\frac{1}{\delta}\im \int U_\delta\, x\,G_0,\\
\Phi_{\delta}^4&=\frac{1}{\delta}\re \int U_\delta\, x\,G_0,&
\Phi_{\delta}^5&=\frac{1}{\delta}\im \int U_\delta\, \left(|x|^2-\frac{N}{2}\right)G_0.&&
\end{align*}
Denote by $\Gamma_{id}=(1,0,0,0)$ the identical transformation. Note that $\Phi_{\delta}(0,\Gamma_{id},G_0)=0$.  Then:
\begin{claim}
\label{C:implicit}
For small $\delta$, there exist $(\theta,\Gamma)$ close to $(0,\Gamma_{id})$ such that $$\Phi_{\delta}\left(\theta_{\delta},\Gamma_{\delta},\frac{1}{\delta} g_{\delta}\right)=0,$$
where $g_\delta$ is the initial condition of the maximizer $v_{\delta}$ defined in step 1.
\end{claim}
We refer to Appendix \ref{A:implicit} for the proof of Claim \ref{C:implicit} which is based on a standard application of the implicit function theorem.

Let ${u}_{\delta}$ be the solution of \eqref{CP} with initial condition
\begin{equation*}
{f}_{\delta}=e^{i\theta_\delta}\Gamma_\delta\left(v_{\delta}\right)_{\restriction t=0}.
\end{equation*}
Then by \eqref{v0n_to_G},
\begin{equation}
\label{u0n_to_G0}
 \lim_{\delta\to \infty} \delta^{-1}\left\|{f}_{\delta}-\delta G_0\right\|_{L^2}=0.
\end{equation}
Furthermore, from the invariance of the $L^2$ and $L^{2+\frac{4}{N}}_{t,x}$ norms by the transformations of the equation, ${u}_{\delta}$ satisfies the maximizer equations \eqref{E:maximizer}.

The fact that $\Phi_{\delta}\left(\theta_{\delta},\Gamma_{\delta},\delta^{-1}g_{\delta}\right)=0$ means that ${f}_{\delta}$ satisfies the orthogonality conditions
\begin{gather}
\label{ortho'}
\im \int \left({f}_{\delta}-\delta G_0\right)\,G_0=0,\quad \int \left({f}_{\delta}-\delta G_0\right) x\,G_0=0,\\
\label{ortho''}
\int \left({f}_{\delta}-\delta G_0\right)\left(|x|^2-\frac{N}{2}\right)\,G_0=0.
\end{gather}
Let $\alpha_{\delta}=\re \int {f}_{\delta} G_{0}$ and $\varphi_{\delta}={f}_{\delta}-\alpha_{\delta} G_{0}$, so that $\re \int \varphi_{\delta}G_0=0$. By \eqref{ortho'} and \eqref{ortho''}, $\varphi_{\delta}$ satisfies the orthogonality conditions \eqref{ortho_intro}. By \eqref{u0n_to_G0}, $\lim_{\delta\to 0} \alpha_{\delta}/\delta=1$, which concludes Step 2.

\smallskip

\noindent\emph{Step 3. Proof of the estimate \eqref{cond_phi}.} In
this step we conclude the proof of Proposition \ref{P:close_G} using
the coercivity of $Q$ (Theorem \ref{T:ortho}). To simplify
notations, we will omit the index $\delta$ and write $u$, $f$,
$\varphi$ and $\alpha$ instead of $u_{\delta}$, $f_{\delta}$,
$\varphi_{\delta}$ and $\alpha_{\delta}$. All estimates stated hold
for small $\delta>0$.

By Claim \ref{C:first_est},
\begin{multline*}
\left|\iint |u|^{2+\frac{4}{N}}\,dt\,dx-\iint \left| e^{i\frac{t}{2}\Delta}
f\,dt\,dx\right|^{2+\frac{4}{N}}\right|\\
\leq C \delta^{1+\frac{4}{N}} \left\|u-e^{i\frac{t}{2}\Delta}
f\right\|_{L^{2+\frac{4}{N}}_{t,x}}\leq C\delta^{2+\frac{8}{N}}.
\end{multline*}
Recalling that $\frac{1}{\alpha}f=G_0(x)+\frac{1}{\alpha}\varphi$
and using the expansion 
of the Strichartz norm, we obtain
\begin{multline*}
\iint \left|u\right|^{2+\frac{4}{N}} \,dt\,dx = \alpha^{2+\frac{4}{N}}
\iint \left|e^{i\frac{t}{2}\Delta}\frac{1}{\alpha}f\right|^{2+\frac{4}{N}}\,dt\,dx
+\OOO\left(\delta^{2+\frac{8}{N}}\right)\\
= C_S\left(\int \left|f\right|^2\,dx \right)^{1+\frac{2}{N}}-
\alpha^{2+\frac{4}{N}}Q\left(\frac{1}{\alpha}\varphi\right)+
\alpha^{2+\frac{4}{N}}\OOO\left(\frac{1}{\alpha^3}\left\|\varphi\right\|_{L^2}^3\right)
+\OOO\left(\delta^{2+\frac{8}{N}}\right)\\
= C_S\delta^{2+\frac{4}{N}}-\alpha^{\frac 4N} Q \left(\varphi\right)
+\OOO\left(\alpha^{\frac{4}{N}-1} \left\|\varphi\right\|_{L^2}^3
\right) +\OOO\left(\delta^{2+\frac{8}{N}}\right).
\end{multline*}
Using that $u$ satisfies \eqref{E:maximizer}, we get
$$
\iint \left|u\right|^{2+\frac{4}{N}}\,dt\,dx = I(\delta)= C_S \delta^{2+\frac{4}{N}}+
\OOO\left(\delta^{2+\frac{8}{N}}\right),
$$
and thus,
$$
\alpha^{\frac 4N} Q\left(\varphi\right)=\OOO\left(\alpha^{
\frac{4}{N}-1}\|\varphi\|_{L^2}^3\right)+\OOO\lf(\delta^{2+
\frac{8}{N}} \rg) =\OOO\left(\alpha^{\frac{4}{N}}\|\varphi\|_{L^2}^2
\frac{\|\varphi\|_{L^2}}{\alpha}\right)+\OOO\lf(\delta^{2+\frac{8}{N}}\rg).
$$
By Theorem \ref{T:ortho}, $\|\varphi\|_{L^2}^2\lesssim
Q\left(\varphi\right)$, and thus, using that
$\frac{1}{\alpha}\lf\|\varphi\rg\|_{L^2}\to 0$ as $\delta\to 0$,
$$
\alpha^{\frac 4N}\|\varphi\|_{L^2}^2=\OOO\lf(\delta^{2+\frac{8}{N}}\rg),
$$
which shows \eqref{cond_phi}.
\end{proof}

\subsection{Proof of the estimate on the maximum}
The idea of the proof of Proposition \ref{P:estimate} is to compare $I(\delta)$ with the $L^{2+\frac{4}{N}}$ norm of $H_{\delta}$, the solution to the nonlinear equation \eqref{CP} with
the Gaussian initial data $\delta G_0$.
We have
$$ \iint \left|u_{\delta}\right|^{2+\frac{4}{N}}\,dt\,dx=I(\delta)\geq \iint \left|H_{\delta}\right|^{2+\frac{4}{N}}\,dt\,dx.$$

The global $L^{2+\frac{4}{N}}$ of $H_{\delta}$ may be estimated as follows:
\begin{lemma}
\label{L:gaussian_devt}
Let
\begin{equation}
 \label{def_c*}
D_N=-\left(2+\frac{4}{N}\right)\im \iint |G(t)|^{\frac{4}{N}} \overline{G}(t) \int_0^te^{i\frac{(t-s)}{2}\Delta}\left(|G(s)|^{\frac{4}{N}}G(s)\right)ds\,dt\,dx.
\end{equation}
Then for small $\delta>0$,
\begin{equation}
\label{gaussian_devt}
\iint |H_{\delta}|^{2+\frac{4}{N}}\,dt\,dx=\delta^{2+\frac 4N}\iint |G|^{2+\frac{4}{N}}\,dt\,dx\\
+\gamma D_N\delta^{2+\frac 8N}
+\OOO\left(\delta^{2+\frac{12}{N}}\right).
\end{equation}
\end{lemma}
The exact value of the constant $D_N$ will be computed in Appendix \ref{A:const1} (dimension $1$) and Appendix \ref{A:const2} (dimension $2$) .
\begin{proof}[Proof of Lemma \ref{L:gaussian_devt}]
Since $G$ is the linear evolution of $G_0$, we have
$$
H_{\delta}=\delta G + i \gamma \int_0^t e^{i\frac{(t-s)}{2}\Delta}|H_{\delta}(s)|^{\frac{4}{N}}H_{\delta}(s)ds.
$$
We approximate $H_\delta$ by $v_\delta$:
$$
v_{\delta}(t,x)=\delta \left(G(t,x)+ \gamma \delta^{\frac4{N}}
r(t,x)\right),
$$
where
\begin{equation}
\label{def_r}
r(t,x)=i\int_0^t e^{i\frac{(t-s)}{2}\Delta} |G(s)|^{\frac4{N}} G(s) \,
ds,
\end{equation}
in other words, $v_\delta$ solves
$$
i \partial_tv_{\delta} + \frac12 \Delta v_{\delta} + \gamma
\delta^{\frac4{N}+1}|G|^{\frac4{N}} G = 0, \quad v_{\delta}(0,x)=\delta
G_0(x),
$$
and $r$ solves
$$
i \partial_t r + \frac12 \Delta r + |G|^{\frac4{N}} G = 0, \quad r(0,x)=0.
$$
Since by Claim \ref{C:first_est}
\begin{multline*}
\left\||H_\delta|^{\frac 4N} H_\delta - \delta^{\frac4{N}+1}
|G|^{\frac4{N}} G \right\|_{L^{\frac{2(N+2)}{N+4}}_{t,x}}\\
\leq C
\left\|H_{\delta}-\delta G\right\|_{L^{2+\frac{4}{N}}_{t,x}}\left(\left\|H_\delta\right\|^{\frac 4N}_{L^{2+\frac{4}{N}}_{t,x}}+\left\|\delta G\right\|_{L^{2+\frac{4}{N}}_{t,x}}^{\frac 4N}\right)
\leq
C\delta^{\frac8{N}+1},
\end{multline*}
by Strichartz estimates, 
we have
$$
\|{H_\delta} - v_{\delta} \|_{L^{2+\frac{4}{N}}_{t,x}} \leq C \delta^{1+\frac8{N}},
$$
and thus,
$$
\bigg|\iint |H_{\delta}|^{2+\frac{4}{N}} \, dt\,dx- \iint
|v_{\delta}|^{2+\frac{4}{N}} \, dt\,dx \bigg| \lesssim
\|H_\delta - v_\delta\|_{L^{2+\frac{4}{N}}_{t,x}} \, \|\delta
G_0\|^{1+\frac4{N}}_{L^{2}} \lesssim \delta^{2+\frac{12}{N}},
$$
which is exactly the power of higher order terms in \eqref{gaussian_devt}. It remains to estimate
$\iint |v_{\delta}|^{2+\frac{4}{N}}$. Note that if $A$ and $B$ are functions of space and time,
\begin{multline}
\label{A+B}
\iint |A+B|^{2+\frac4N}\\
=\iint |A|^{2+\frac{4}{N}}+\left(2+\frac 4N\right)\re \iint |A|^{\frac 4N}A\overline{B}+\OOO\left(\iint |A|^{\frac{4}{N}}|B|^2+|B|^{2+\frac 4N}\right).
\end{multline}
By \eqref{A+B} and the definition of $v_{\delta}$ we get,
\begin{multline*}
\iint |v_{\delta}|^{2+\frac{4}{N}}\,dt\,dx=\delta^{2+\frac 4N}\iint |G|^{2+\frac{4}{N}}\,dt\,dx\\
+\delta^{2+\frac 8N}\left(2+\frac 4N\right)\re \iint \left|G\right|^{\frac{4}{N}}\overline{G}r\,dt\,dx
+\OOO\left(\delta^{2+\frac{12}{N}}\right),
\end{multline*}
which concludes the proof of Lemma \ref{L:gaussian_devt} in view of the definition \eqref{def_c*} of $D_N$.
\end{proof}
We next prove Proposition \ref{P:estimate}. Let $u_{\delta}$,
$f_{\delta}$, $\varphi_{\delta}$ and $\alpha_{\delta}$ be as in
Proposition \ref{P:close_G}. We have
\begin{equation}
 \label{expr_un}
u_{\delta}= \underbrace{e^{i\frac{t}{2}\Delta}\left(\alpha_{\delta}
G_0+\varphi_{\delta}\right)}_A+ \underbrace{i\gamma
\int_0^te^{i\frac{(t-s)}{2}\Delta}\left(|u_{\delta}(s)|^{\frac
4N}u_{\delta}(s)\right)ds}_B.
\end{equation}
By \eqref{cond_phi} and Strichartz estimate \eqref{strichartz},
$$
\left\|e^{i\frac t2\Delta}\varphi_\delta \right\|_{L^{2
+\frac{4}{N}}}\leq C\|\varphi_{\delta}\|_{L^2}\leq C\delta^{1+\frac{2}{N}}.
$$
Expanding the $B$ term in \eqref{expr_un} and applying Strichartz
estimates again to bound the terms in $\varphi_{\delta}$, we get 
(the $\OOO$'s are estimated in the space $L^{2+\frac{4}{N}}_{t,x}$).
\begin{multline*}
B=i\gamma\int_0^te^{i\frac{(t-s)}{2}\Delta}\left(|u_{\delta}(s)|^{\frac 4N}u_{\delta}(s)\right)ds\\
=i\gamma\int_0^te^{i\frac{(t-s)}{2}\Delta}\left[\left|\alpha_{\delta} G(s)+e^{i\frac s2 \Delta}\varphi_{\delta} \right|^{\frac 4N}\left(\alpha_{\delta} G(s)+e^{i\frac s2 \Delta}\varphi_{\delta}\right)  \right]ds+\OOO\left(\delta^{1+\frac{8}{N}}\right)\\
=i\gamma\alpha_{\delta}^{1+\frac 4N}\int_0^te^{i\frac{(t-s)}{2}\Delta}\left(\left|G(s)\right|^{\frac 4N} G(s)  \right)ds+\OOO\left(\delta^{\frac 4N}\|\varphi_{\delta}\|_{L^2}+\delta^{1+\frac{8}{N}}\right).
\end{multline*}
And thus, by \eqref{A+B} and \eqref{expr_un},
\begin{multline}
\label{expr_iint}
\iint |u_{\delta}|^{2+\frac{4}{N}}\,dt\,dx
=
\iint \left|\alpha_{\delta}
G+e^{i\frac t2\Delta}\varphi_{\delta}\right|^{2+\frac{4}{N}}\,dt\,dx
\\-\left(2+\frac 4N\right)\gamma \alpha_{\delta}^{2+\frac{8}{N}}\im \iint |G(t)|^{\frac{4}{N}} \overline{G}(t) \int_0^te^{i\frac{(t-s)}{2}\Delta}\left(|G(s)|^{\frac{4}{N}}G(s)\right)ds\,dt\,dx
\\+\OOO\left(\delta^{1+\frac{8}{N}}\|\varphi_{\delta}\|_{L^2}\right)+\OOO\left(\delta^{2+\frac{12}{N}}\right).
\end{multline}
By the equation \eqref{devt_Q}
\begin{multline*}
\iint \left|\alpha_{\delta}
G+e^{i\frac t2\Delta}\varphi_{\delta}\right|^{2+\frac{4}{N}}\,dt\,dx\\
=\alpha_{\delta}^{2+\frac{4}{N}}\left[C_S\left\|G_0+\frac{1}{\alpha_{\delta}}\varphi_{\delta}   \right\|^{2+\frac{4}{N}}_{L^2}-Q\left(\frac{1}{\alpha_{\delta}}\varphi_{\delta}\right)+\OOO\left(\frac{1}{\alpha_{\delta}^3}\left\|\varphi_{\delta}\right\|_{L^2}^3\right)\right].
\end{multline*}
By \eqref{cond_phi} and \eqref{expr_iint}, using
that
\begin{equation}
\label{compare_alpha_delta}
\left\|\alpha_{\delta} G_0+\varphi_{\delta}\right\|_{L^2}^2=\delta^2=
\alpha_{\delta}^2+\|\varphi_{\delta}\|^2_{L^2}=\alpha_{\delta}^2+\OOO\left(\delta^{2+\frac 4N}\right),
\end{equation}
we get, in view of the definition \eqref{def_c*} of $D_N$,
\begin{multline}
\label{est_I_interm}
\iint |u_{\delta}|^{2+\frac{4}{N}}\,dt\,dx\\
=C_S\delta^{2+\frac{4}{N}}-\alpha_{\delta}^{\frac{4}{N}}Q(\varphi_{\delta})+\gamma D_N
\alpha_{\delta}^{2+\frac{8}{N}}+\OOO\left(\delta^{1+\frac{8}{N}}\left\|\varphi_{\delta}\right\|_{L^2}\right)+\OOO\left(\delta^{2+\frac{12}{N}}\right).
\end{multline}
By Lemma \ref{L:gaussian_devt},
$$ \iint |u_{\delta}|^{2+\frac{4}{N}}\,dt\,dx\geq C_S\delta^{2+\frac{4}{N}}+\gamma D_N\delta^{2+\frac{8}{N}}+\OOO\left(\delta^{2+\frac{12}{N}}\right).$$
Combining with \eqref{est_I_interm}, we get
\begin{multline*}
C_S\delta^{2+\frac{4}{N}}-\alpha_{\delta}^{\frac{4}{N}}Q(\varphi_{\delta})+\gamma D_N
\alpha_{\delta}^{2+\frac{8}{N}}+\OOO(\delta^{2+\frac{12}{N}})+\OOO\left(\delta^{1+\frac{8}{N}}\left\|\varphi_{\delta}\right\|_{L^2}\right)\\
\geq C_S\delta^{2+\frac{4}{N}}+\gamma D_N\delta^{2+\frac{8}{N}}.
\end{multline*}
Using that by \eqref{compare_alpha_delta}
$$\left|\delta^{2+\frac{8}{N}}-\alpha_{\delta}^{2+\frac{8}{N}}\right|=\OOO\left(\delta^{2+\frac{12}{N}}\right),$$ 
this simplifies to
$$   \alpha_{\delta}^{\frac 4N}Q(\varphi_{\delta})=\OOO\left(\delta^{2+\frac{12}{N}}\right)+\OOO\left(\delta^{1+\frac{8}{N}}\|\varphi_{\delta}\|_{L^2}\right).$$
Let $X=\|\varphi_{\delta}\|_{L^2}\delta^{-1-\frac{4}{N}}$. By the preceding estimate and Theorem \ref{T:ortho}, there exists a constant $C>0$ independent of $\delta$ such that
$X^2\leq C(1+X)$. This implies that $X$ is bounded independently of $\delta$, i.e.
\begin{equation}
\label{new_varphi}
\|\varphi_{\delta}\|_{L^2}=\OOO\left(\delta^{1+\frac{4}{N}}\right).
\end{equation}
By \eqref{est_I_interm} again,
\begin{equation}
 \label{estimate_I}
I(\delta)=\iint\left|u_{\delta}\right|^{2+\frac{4}{N}}=C_S\delta^{2+\frac{4}{N}}+\gamma D_N
\delta^{2+\frac{8}{N}}+\OOO\left(\delta^{2+\frac{12}{N}}\right).
\end{equation}
The proof is complete, except for the computation of $D_N$ which is given in appendices \ref{A:const1} and \ref{A:const2}. Note that as announced in Remark \ref{R:cond_phi}, the estimate \eqref{new_varphi} improves the preceding estimate \eqref{cond_phi} on $\varphi_{\delta}$.
\qed

\section{Uniqueness}
\label{S:uniqueness}
In this section we show the uniqueness part of Theorem \ref{T:main}. We assume again
$$ N\in \{1,2\}.$$
By Proposition \ref{P:maximizer}, there exists, for small $\delta>0$, a maximizer for $I(\delta)$, i.e. a solution $u_{\delta}$ of \eqref{CP} such that  
\begin{equation}
\label{H_maximizer}
\|f_{\delta}\|_{L^2}=\delta,\quad \iint |u_{\delta}|^{2+\frac{4}{N}}=I(\delta)
\end{equation}
(as usual $f_{\delta}(x)=u_{\delta}(0,x)$).
By Proposition \ref{P:close_G} and Remark \ref{R:cond_phi}, assuming again that $\delta$ is small, any maximizer for $I(\delta)$ satisfies, after transformation, the following properties:
\begin{equation}
\label{H_close_to_G}
f_{\delta} = \alpha_{\delta} G_0+\varphi_{\delta},
\end{equation}
where $\varphi_{\delta}\in L^2(\RR^N)$ and $\alpha_{\delta}>0$ are such that
\begin{gather}
\label{H_ortho}
\int \varphi\, G_0=\int \varphi\, |x|^2 G_0=0, \quad \int \varphi\,
x G_0=0_{\RR^N},\\
\label{H_estim}
 \|\varphi_{\delta}\|_{L^2}\leq C\delta^{1+\frac 4N},\quad \alpha_{\delta}>0\quad \text{and}\quad
\left|\alpha_{\delta}-\delta\right|\leq C\delta^{1+\frac 4N}.
\end{gather} 
We must show that if $C>0$, there exists $\delta_0>0$ such that if $\delta\in (0,\delta_0)$, there is at most one solution $u_{\delta}$ of \eqref{CP} satisfying \eqref{H_maximizer}, \eqref{H_close_to_G}, \eqref{H_ortho} and \eqref{H_estim}.

Let us fix a small $\delta>0$ and a maximizer $u_{\delta}$ satisfying \eqref{H_maximizer}, \eqref{H_close_to_G}, \eqref{H_ortho} and \eqref{H_estim}. The strategy of the proof is to expand $\int |v|^{2+\frac{4}{n}}$, where $v$ is a solution of \eqref{CP} which is close to $u_{\delta}$. In \S \ref{SS:linearization} we expand $v$ and $\int |v|^{2+\frac{4}{n}}$ at first order, in \S \ref{SS:second} we obtain a second order expansion involving the quadratic form $Q$. Assuming that $v$ is another maximizer, the conclusion will follow from Theorem \ref{T:ortho}.

\subsection{Linearization}
\label{SS:linearization}
\begin{lemma}
\label{L:linearization}
There exists a linear operator $L_{\delta}:L^{2+\frac{4}{N}}_{t,x}\to L^{2+\frac 4N}_{t,x}$ such that
\begin{equation}
\label{estim_L_delta}
\forall h\in L^{2+\frac{4}{N}}_{t,x},\quad \|(L_{\delta} -1)h\|_{L^{2+\frac{4}{N}}_{t,x}}\leq C\delta^{\frac{4}{N}} \|h\|_{L^{2+\frac{4}{N}}_{t,x}},
\end{equation} 
with the following property:
if $v$ is a solution of \eqref{CP} with the initial condition $f_{\delta}+\psi$, where
\begin{equation}
\label{small_psi}
 \|\psi\|_{L^2}\leq \delta,
\end{equation} 
then
\begin{equation}
\label{estim_u_v_psi}
\left\|v-u_{\delta}-L_{\delta}\left(e^{i\frac{t}{2}\Delta}\psi\right)\right\|_{L^{2+\frac 4N}_{t,x}}\leq C\delta^{\frac{4}{N}-1}\|\psi\|_{L^2}^2.
\end{equation} 
\end{lemma}
\begin{proof}
 Let $w=v-u_{\delta}$. Then by Lemma \ref{L:longperturb}, 
\begin{equation}
\label{bound_w_psi}
 \|w\|_{L^{2+\frac{4}{N}}_{t,x}}\leq C\|\psi\|_{L^2}.
\end{equation} 
Writing Duhamel's formula for $u_{\delta}$ and $v=u_\delta+w$, we get
\begin{equation*}
w=e^{i\frac{t}{2}\Delta}\psi+i\gamma \int_0^te^{i\frac{(t-s)}{2}\Delta}\left(|u_{\delta}(s)+w(s)|^{\frac{4}{N}}(u_\delta(s)+w(s))-|u_{\delta}(s)|^{\frac{4}{N}}u_{\delta}(s)\right)ds.
\end{equation*} 
Expanding $|u_{\delta}(s)+w(s)|^{\frac{4}{N}}(u_\delta(s)+w(s))$, one can write the preceding equality as
\begin{equation}
\label{expr_w}
w=e^{i\frac{t}{2}\Delta}\psi+\widetilde{L_{\delta}} w+\widetilde{R_{\delta}}(w),
\end{equation} 
where the linear operator $\widetilde{L_{\delta}}:L^{2+\frac{4}{N}}_{t,x}\to L^{2+\frac{4}{N}}_{t,x}$ satisfies
\begin{equation}
\label{tilde_L}
\left\|\widetilde{L_{\delta}}w\right\|_{L^{2+\frac{4}{N}}_{t,x}}\leq C\delta^{\frac{4}{N}}\left\|w\right\|_{L^{2+\frac{4}{N}}_{t,x}}, 
\end{equation} 
and $\widetilde{R_{\delta}}$ satisfies
\begin{equation}
\label{bound_tilde_R}
 \left\|\widetilde{R_{\delta}}(w)\right\|\leq C\left(\delta^{\frac 4N-1}\|w\|_{L^{2+\frac{4}{N}}_{t,x}}^2+\|w\|_{L^{2+\frac{4}{N}}_{t,x}}^{1+\frac{4}{N}}\right).
\end{equation} 
Letting for small $\delta$
$$ L_{\delta}=\left(1-\widetilde{L_{\delta}}\right)^{-1},$$
we obtain by \eqref{tilde_L} that $L_{\delta}$ satisfies \eqref{estim_L_delta}. The estimate \eqref{estim_u_v_psi} follows from \eqref{small_psi}, \eqref{bound_w_psi}, \eqref{expr_w} and \eqref{bound_tilde_R}.
\end{proof}
\begin{lemma}
\label{L:lagrange}
Let $L_\delta$ be as in Lemma \ref{L:linearization}. Then for small $\delta>0$,
\begin{equation}
\label{NL_lagrange}
\re \iint |u_{\delta}|^{\frac 4N}\overline{u_\delta} L_{\delta}\left(e^{i\frac{t}{2}\Delta}\psi\right)= \mu_{\delta}\re \int \overline{f_{\delta}}\psi,
\end{equation} 
where $\mu_{\delta}>0$, which depends only on $u_{\delta}$, satisfies
\begin{equation}
\label{estimate_mu}
\left|\mu_{\delta}-C_S\delta^{\frac 4N}\right|\leq C\delta^{\frac 8N}.
\end{equation} 
\end{lemma}
\begin{proof}
Indeed, by definition 
\begin{equation}
\label{maximizing}
I(\delta)=\max \iint |v|^{2+\frac{4}{N}},
\end{equation} 
where the maximum is taken over all solutions $v$ of \eqref{CP} with initial condition $f_{\delta}+\psi$, such that $\int |f_{\delta}+\psi|^2=\delta^2$. For such a solution $v$, write, as in the proof of Lemma \ref{L:linearization}, $v=u_{\delta}+w$. Then
\begin{multline*}
\iint |v|^{2+\frac{4}{N}}=\iint |u_{\delta}+w|^{2+\frac{4}{N}}\\
=\iint |u_{\delta}|^{2+\frac{4}{N}}+\left(2+\frac 4N\right)\re \iint |u_{\delta}|^{\frac 4N}\overline{u_\delta} w
+\OOO\left(\delta^{\frac{4}{N}}\|\psi\|_{L^2}^{2}\right)\\
=\iint |u_{\delta}|^{2+\frac{4}{N}}+\left(2+\frac 4N\right)\re \iint |u_{\delta}|^{\frac 4N}\overline{u_\delta} L_{\delta}\left(e^{i\frac t2\Delta}\psi\right)
+\OOO\left(\delta^{\frac{4}{N}}\|\psi\|_{L^2}^{2}\right).
\end{multline*}
The existence of $\mu_{\delta}$ then follows from the Lagrange multiplier equation for the maximizing problem \eqref{maximizing}.

We next estimate $\mu_{\delta}$. 
By \eqref{H_close_to_G} and \eqref{H_estim} 
$$ f_{\delta}=\delta G_0+ \OOO\left(\delta^{1+\frac{4}{N}}\right)\text{ in }L^{2}.$$
Thus by Claim \ref{C:first_est},
\begin{equation}
\label{u_delta_G} 
u_{\delta}=\delta G+ \OOO\left(\delta^{1+\frac{4}{N}}\right)\text{ in }L^{2+\frac 4N}_{t,x}.
\end{equation} 
As a consequence, we obtain (assuming $\|\psi\|_{L^2}\leq \delta$)
\begin{multline*}
\re \iint |u_{\delta}|^{\frac 4N}\overline{u}_\delta L_{\delta}\left(e^{i\frac t2\Delta}\psi\right)=
\re \iint |u_{\delta}|^{\frac 4N}\overline{u}_\delta\, e^{i\frac{t}{2}\Delta}\psi+\OOO\left(\delta^{1+\frac{8}{N}}\|\psi\|_{L^2}\right)\\
= \delta^{1+\frac{4}{N}}\re \iint |G|^{\frac 4N}\overline{G}\, e^{i\frac{t}{2}\Delta}\psi+\OOO\left(\delta^{1+\frac{8}{N}}\|\psi\|_{L^2}\right).
\end{multline*}
On the other hand,
$$ \re\int \overline{f}_{\delta}\psi=\re\int \alpha_{\delta} G_0\psi+\OOO\left(\delta^{1+\frac{4}{N}}\|\psi\|_{L^2}\right)=\delta\re\int  G_0\psi+\OOO\left(\delta^{1+\frac{4}{N}}\|\psi\|_{L^2}\right).$$
Combining with \eqref{NL_lagrange}, we get
$$\delta^{1+\frac{4}{N}}\re \iint |G|^{\frac 4N}\overline{G} e^{i\frac{t}{2}\Delta}\psi
=\delta \mu_{\delta}\re\int G_0\psi+\OOO\left(\delta^{1+\frac{8}{N}}\|\psi\|_{L^2}+\mu_{\delta}\delta^{1+\frac{4}{N}}\|\psi\|_{L^2}\right).$$
By \eqref{linear_lagrange},
$$C_S\delta^{1+\frac{4}{N}}\re \int G_0\psi=\delta \mu_{\delta}\re\int G_0\psi+\OOO\left(\mu_{\delta}\delta^{1+\frac{4}{N}}\|\psi\|_{L^2}+\delta^{1+\frac{8}{N}}\|\psi\|_{L^2}\right).$$
This holds for all small $\psi\in L^2$, yielding \eqref{estimate_mu}.
\end{proof}

\subsection{Second order expansion}
\label{SS:second}
\begin{lemma}
\label{L:NL_devt}
Let $v$ be a solution of \eqref{CP} with initial condition $f_{\delta}+\psi$, and assume
$$ \int |f_{\delta}+\psi|^2=\delta^2.$$
Then
\begin{equation}
\label{NL_devt}
\iint |v|^{2+\frac{4}{N}}=I(\delta)-\delta^{\frac 4N}Q(\psi)+\OOO\left(\delta^{\frac 8N}\|\psi\|^{2}_{L^2}+\delta^{\frac 4N-1}\|\psi\|^3_{L^2}+\|\psi\|_{L^2}^{2+\frac 4N}\right).
\end{equation}
\end{lemma}
\begin{proof}
Using that $\int |f_{\delta}|^2=\delta^2$, we get
\begin{equation}
\label{eq_norm_psi}
\int |\psi|^2=-2\re \int \overline{f_{\delta}}\psi,
\end{equation} 
and thus by \eqref{H_close_to_G} and \eqref{H_estim},
\begin{equation}
\label{almost_ortho}
\delta^2\left|\re \int G_0\psi\right|^2\leq C\left(\delta^{\frac{8}{N}+2}\|\psi\|^2_{L^2}+\|\psi\|^4_{L^2}\right).
\end{equation}
Expanding $|u_{\delta}+w|^{2+\frac{4}{N}}$ at second order in $w$, we obtain
\begin{multline*}
\iint |v|^{2+\frac{4}{N}}=\iint |u_{\delta}+w|^{2+\frac{4}{N}}\\
=\iint |u_{\delta}|^{2+\frac{4}{N}}+\left(2+\frac 4N\right)\re \iint |u_{\delta}|^{\frac 4N}\overline{u}_\delta w\\
+\left(1+\frac 2N\right)^2\iint |u_{\delta}|^{\frac{4}{N}}|w|^2+\frac{2}{N}\left(1+\frac{2}{N}\right)\re\iint |u_{\delta}|^{\frac 4N-2}\overline{u}_{\delta}^2w^2\\+\OOO\left(\delta^{\frac 4N-1}\|\psi\|^3_{L^2}\right)+\OOO\left(\|\psi\|_{L^2}^{2+\frac 4N}\right).
\end{multline*}
By Lemma \ref{L:linearization}, Lemma \ref{L:lagrange} and \eqref{eq_norm_psi}, 
\begin{multline*}
 \re \iint |u_{\delta}|^{\frac 4N}\overline{u}_\delta w
=\re \iint |u_{\delta}|^{\frac 4N}\overline{u}_\delta L_{\delta}─\left(e^{i\frac t2 \Delta}\psi\right)+\OOO\left(\delta^{\frac{8}{N}}\|\psi\|^2_{L^2}\right)\\
=\mu_{\delta}\int \overline{f}_{\delta}\psi+\OOO\left(\delta^{\frac{8}{N}}\|\psi\|^2_{L^2}\right)=-\frac{\mu_{\delta}}{2}\int |\psi|^2+\OOO\left(\delta^{\frac{8}{N}}\|\psi\|^2_{L^2}\right)\\
=-\frac{C_S}{2}\delta^{\frac 4N}\int |\psi|^2+\OOO\left(\delta^{\frac{8}{N}}\|\psi\|^2_{L^2}\right).
\end{multline*}
By \eqref{u_delta_G}, then Lemma \ref{L:linearization},
\begin{multline*}
\iint |u_{\delta}|^{\frac{4}{N}}|w|^2= \delta^{\frac{4}{N}}\iint |G|^{\frac{4}{N}}|w|^2+\OOO\left(\delta^{\frac{8}{N}}\|\psi\|^2_{L^2}\right)\\
=\delta^{\frac{4}{N}}\iint |G|^{\frac{4}{N}}\left|e^{i\frac{t}{2}\Delta}\psi\right|^2+\OOO\left(\delta^{\frac{8}{N}}\|\psi\|^2_{L^2}\right),
\end{multline*}
and similarly
\begin{equation*}
 \re\iint |u_{\delta}|^{\frac 4N-2}\overline{u}_{\delta}^2w^2=\delta^{\frac{4}{N}}\re\iint |G|^{\frac 4N-2}\overline{G}^2\left(e^{i\frac{t}{2}\Delta}\psi\right)^2+\OOO\left(\delta^{\frac{8}{N}}\|\psi\|^2_{L^2}\right).
\end{equation*} 
Combining the preceding estimates, we obtain 
\begin{multline*}
\iint |v|^{2+\frac{4}{N}}
=\iint |u_{\delta}|^{2+\frac{4}{N}}-C_S\left(\frac{N+2}N\right)\delta^{\frac 4N}\int |\psi|^2\\
+\left(1+\frac 2N\right)^2\delta^{\frac{4}{N}}\iint  |G|^{\frac{4}{N}}\left|e^{i\frac{t}{2}\Delta}\psi\right|^2+\frac{2}{N}\left(1+\frac{2}{N}\right)\delta^{\frac{4}{N}}\re\iint |G|^{\frac 4N-2}\overline{G}^2\left(e^{i\frac t2\Delta}\psi\right)^2\\+\OOO\left(\delta^{\frac 8N}\|\psi\|^{2}_{L^2}+\delta^{\frac 4N-1}\|\psi\|^3_{L^2}+\|\psi\|_{L^2}^{2+\frac 4N}\right),
\end{multline*}
which yields \eqref{NL_devt} in view of \eqref{almost_ortho} and the definition \eqref{defQ} of $Q$.
\end{proof}
We can now conclude the proof of the uniqueness of the maximizer. Assume that $\delta>0$ is small and consider a solution $\tilde{u}_{\delta}$ of \eqref{CP} with initial condition $\tilde{f}_{\delta}=\tilde{\alpha}_{\delta}G_0+\tilde{\varphi}_{\delta}$. Assume that $\tilde{u}_{\delta}$, $\tilde{f}_{\delta}$, $\tilde{\varphi}_{\delta}$ and $\tilde{\alpha}_{\delta}$ also satisfy  \eqref{H_maximizer}, \eqref{H_close_to_G}, \eqref{H_ortho} and \eqref{H_estim}. We must show that $\tilde{u}_{\delta}=u_{\delta}$. 
Let
$$ \psi=(\tilde{\alpha}_{\delta}-{\alpha}_{\delta})G_0+\tilde{\varphi}_{\delta}-\varphi_{\delta}.$$
By \eqref{H_estim}, $\|\psi\|_{L^2}\leq C\delta^{\frac{4}{N}+1}$. 
By Lemma \ref{L:NL_devt} with $v=\tilde{u}_{\delta}$, 
\begin{equation*}
I(\delta)=\iint |v|^{2+\frac{4}{N}}=I(\delta)-\delta^{\frac 4N}Q(\psi)+\OOO\left(\delta^{\frac 8N}\|\psi\|^{2}_{L^2}+\delta^{\frac 4N-1}\|\psi\|^3_{L^2}+\|\psi\|_{L^2}^{2+\frac 4N}\right),
\end{equation*}
and thus,
\begin{equation}
\label{uniqueness1}
\delta^{\frac{4}{N}}Q(\psi)\leq C\left(\delta^{\frac 8N}\|\psi\|^{2}_{L^2}+\delta^{\frac 4N-1}\|\psi\|^3_{L^2}+\|\psi\|_{L^2}^{2+\frac 4N}\right)\leq C \delta^{\frac 8N}\|\psi\|^{2}_{L^2}. 
\end{equation} 
Since $G_0$ is in the kernel of $Q$, $Q(\psi)=Q(\tilde{\varphi}_{\delta}-\varphi_{\delta})$. Using that $\varphi_{\delta}$ and $\tilde{\varphi}_{\delta}$ satisfy the orthogonality conditions \eqref{ortho_intro}, we deduce from Theorem \ref{T:ortho}:
\begin{equation}
\label{uniqueness2}
c\|\tilde{\varphi}_{\delta}-\varphi_{\delta}\|^2\leq Q(\psi).
\end{equation} 
Using that 
$$\alpha^2_{\delta}+\int |\varphi_{\delta}|^2=\delta^2=\tilde{\alpha}^2_{\delta}+\int |\tilde{\varphi}_{\delta}|^2,$$
we obtain, in view of \eqref{H_estim}, 
\begin{equation*}
|\tilde{\alpha}_{\delta}-\alpha_{\delta}|=\left|\frac{\tilde{\alpha}^2_{\delta}-\alpha_{\delta}^2}{\alpha_{\delta}+\tilde{\alpha}_{\delta}}\right|\\
\leq \frac{1}{\delta}\left|\int |\varphi_{\delta}|^2-\int |\tilde{\varphi}_{\delta}|^2\right|
\leq C\delta^{\frac 4N}\|\varphi_{\delta}-\tilde{\varphi}_{\delta}\|_{L^2},
\end{equation*}
and thus for small $\delta$,
\begin{equation}
\label{uniqueness3}
\|\psi\|_{L^2}^2=(\tilde{\alpha}_{\delta}-\alpha_{\delta})^2+\|\varphi_{\delta}-\tilde{\varphi}_{\delta}\|_{L^2}^2\leq 2\|\varphi_{\delta}-\tilde{\varphi}_{\delta}\|_{L^2}^2. 
\end{equation} 
Combining \eqref{uniqueness1}, \eqref{uniqueness2} and \eqref{uniqueness3}, we get
$$ \delta^{\frac{4}{N}}\|\psi\|^2_{L^2}\leq C\delta^{\frac 8N}\|\psi\|^{2}_{L^2},$$
a contradiction if $\delta>0$ is small and $\psi\neq 0$. Thus, $\psi=0$ and $u_{\delta}=\tilde{u}_{\delta}$, which completes the proof.

\section{Coercivity of the quadratic form}
\label{S:quadratic}

In this section we show Theorem \ref{T:ortho}.

Let $F$ be the $N+2$-dimensional space of the null directions for $Q$ that are generated by the continuous symmetries of the linear Schr\"odinger equation:
$$ F=\vect_{\CC} \{G_0,x_jG_0,|x|^2G_0\}$$
($j=1$ or $j=1,2$ in dimension $1$ and $2$ respectively).

We must show that there exists a constant $c>0$ such that
$$ \varphi\in F^{\bot}\Longrightarrow Q(\varphi)\geq c\|\varphi\|_{L^2}^2.$$
It turns out that $F$ is generated by eigenfunctions for the harmonic oscillator defined in  \S \ref{SS:harmonic}. Indeed, in dimension $1$, $F$ is spanned by $h_0$, $h_1$ and $h_2$ and in dimension $2$ by $h_{00}$, $h_{10}$, $h_{01}$ and $h_{20}+h_{02}$.

The outline of this section is as follows. In \S \ref{S:preliminaries} we recall some properties of the harmonic oscillator $\HHH=-\Delta+|x|^2$ and of a lens transform that will be used in the proof. In \S \ref{SS:reduction} we
show that the proof of Theorem \ref{T:ortho} reduces to the proof that $Q(\varphi)>0$ for any
eigenfunction $\varphi$ of the harmonic oscillator $\HHH$ that is
orthogonal to $F$. In \S \ref{SS:1d} and \S \ref{SS:2d} we treat the
reduced problem in $1D$ and $2D$ respectively by estimating the
values taken by the quadratic form on the eigenfunctions of $\HHH$.

\subsection{Preliminaries}
\label{S:preliminaries}

\subsubsection{Harmonic oscillator}
\label{SS:harmonic}

Consider the linear Schr\"odinger equation with the harmonic
potential:
\begin{gather}
 \label{E:LS-harm}
i \partial_{\tau} u - \frac12 \mathcal{H} u =0, \quad (\tau,y)\in \RR\times \RR^N,\\
\notag\text{where }
\mathcal{H}=-\Delta+|y|^2.
\end{gather}
In what follows we briefly recall spectral property of $\HHH$. We
refer to \cite{Carles09} and references therein for more details.

We first review the spectral properties of $\mathcal{H}$ in one
space dimension. The spectrum of $\mathcal{H}$ consists of positive
eigenvalues $\lambda_n=2n+1$, $n=0,1,...$, and the corresponding
eigenfunctions are
\begin{equation}
 \label{E:herm1}
h_n(y) = (-1)^n \, c_n \, e^{y^2/2} \, \partial_y^n(e^{-y^2}), \quad
c_n = \frac1{\sqrt{n!} \, 2^{n/2}},
\end{equation}
here the coefficients $c_n$ are chosen so that $\|h_n\|^2_{L^2(\cR)}
= \sqrt \pi$. Equivalently, these are the Hermite functions
\begin{equation}
 \label{E:herm2}
h_n(y) = \frac{H_n(y)}{\sqrt{2^n \, n!}} \, e^{-y^2/2}, \
\end{equation}
with $H_n(y)$ being the $n^{th}$ Hermite polynomial:
$$
H_n(y) = (-1)^n \, e^{y^2} \, \partial_y^n(e^{-y^2}).
$$
Thus, $H_0(y) = 1$, $H_1(y)=2y$, $H_2(y)=4y^2-2$, $H_3(y)=8y^3-12y$, $H_4(y)=16y^4-48y^2+12$,
etc. These eigenfunctions are orthogonal
\begin{equation}
 \label{E:herm3}
\int_\cR h_j(y) \, h_k(y) \,dy = \frac{1}{\sqrt{2^j \, j!}\sqrt{2^k
\, k!}}\int_\cR H_j(y) \, H_k(y) \, e^{-y^2}\, dy = \sqrt \pi \,
\delta_{jk},
\end{equation}
and they span $L^2(\cR)$. 

In the $2D$ set up, $y=(y_1,y_2) \in \cR^2$, the spectrum of
$\mathcal{H}$ consists as well of a discrete set of positive
eigenvalues $\{\lambda_n\}_{n\in \NN}$ and, for $n \in \NN$, one has
\begin{equation*}
\lambda_{n} = 2n+2.
\end{equation*}
To each eigenvalue $\lambda_n$ there corresponds a set of
eigenfunctions $h_{jk}(y)$ with the property that $j+k=n$ and
$
h_{jk}(y) = h_j(y_1)h_k(y_2),
$
where the $h_n$'s are the one-dimensional eigenfunctions.
For example, $h_{00}(y) = e^{-|y|^2}$ is the only eigenfunction
corresponding to the smallest eigenvalue $\lambda_0=2$. For
$\lambda_1=4$, the eigenfunctions are
$$
h_{10}(y)=\sqrt 2 \,y_1\, e^{-|y|^2/2} \quad \text{and} \quad
h_{01}(y) = \sqrt 2 \,y_2\, e^{-|y|^2/2},
$$
for $\lambda_2=6$, they are
$$
h_{20}(y)=2^{-1/2} (2y_1^2-1) e^{-|y|^2/2}, \quad h_{02}(y)=2^{-1/2}
(2y_2^2-1) e^{-|y|^2/2}
$$
$$
\text{and} \quad h_{11}(y)={2} y_1 y_2 e^{-|y|^2/2}.
$$

\subsubsection{The Lens transform}
\label{SS:lens}

For a function $u(t,x): I \times \cR^N \to \cC$, define the {\it
lens transform}\footnote{We use the name 'lens transform' as in \cite{Tao09NY} but it should not be confused
with the pseudo-conformal inversion \eqref{PC} of Talanov which is sometimes also called the lens transform.} $\LL u$ of $u$ by
$$
\LL u (\tau, y) = \frac1{\cos^{N/2} \tau} \, u\left(\tan \tau,
\frac{y}{\cos \tau}\right) \, e^{-i |y|^2\, \frac{\tan \tau}2}.
$$
The new variables $(\tau,y)$ are defined by $t = \tan \tau$ and
$x=\frac{y}{ \cos \tau}$, $\tau\in (-\frac{\pi}{2},\frac{\pi}{2})$, and thus,
$\LL u: \tan^{-1}(I)\cap (-\frac{\pi}{2},\frac{\pi}{2}) \times \cR^N \to \cC$. If $I
= \cR$, then $\LL u:(-\frac{\pi}2, \frac{\pi}2) \times \cR^N \to
\cC$: the lens transform compactifies the time. For more details
see for example \cite{Carles02}, \cite{Tao09NY} and reference therein.

If $u(t,x)$ solves \eqref{CP} (for some $\gamma\in \RR$), then
$v=\LL u (\tau,y)$ solves
\begin{equation}
 \label{E:harm4}
\ds i \partial_\tau v - \frac12 \HHH v = -\gamma |v|^{\frac 4N} v,
\end{equation}
and vice versa.

The lens transform preserves the initial data $(\LL u)(0) = u(0)$, and thus, the mass of the solution:
 $$\|(\LL u)(0)\|_{L^2} = \|u(0)\|_{L^2}.$$
Furthermore, all Strichartz norms are also preserved, in particular:
$$
\|\LL u\|_{L^{\frac{4}{N}+2}_{t,x}((-\pi/2,\pi/2)\times \RR^N)} =
\|u\|_{L^{\frac{4}{N}+2}_{t,x}(\RR\times\RR^N)}.
$$

{\bf Example.} Let $G_0 = \frac{1}{\pi^{N/4}}e^{-|x|^2/2}$. The
solution to the linear Schr\"odinger equation \eqref{CPlin} is given
by \eqref{def_G}. The definition of $\LL$ shows that the solution
$e^{-i \, \frac{\tau}{2} \, \mathcal{H}}G_0$ of \eqref{E:LS-harm}
is given by
$$
\widetilde{G}(\tau,y) = \frac{1}{\pi^{N/4}} \,
e^{-i\frac{N}{2}\tau} \, e^{-|y|^2/2}=(\LL G)(\tau,y),
$$
which is consistent with the fact that $G_0$ is an eigenfunction for
the eigenvalue $\lambda_0 = N$ of $\HHH$ (in dimension $N = 1, 2$).

For later use we note that using the invariance of the initial condition and the $L^{\frac 4N+2}$
norm by the lens transform $\LL$, we can rewrite the 
definition \eqref{defQ} of the quadratic form as
\begin{multline}
 \label{defQH}
Q(\varphi) = C_S\left[\frac{N+2}{N}\int |\varphi|^2+\frac{4(N+2)}{N^2}
\left(\re \int G_0\varphi\right)^2\right]\\
-\frac{(N+2)^2}{N^2}\int_{-\frac{\pi}{2}}^{\frac{\pi}{2}}\int_{\RR^N}
G_0^{\frac{4}{N}}\left|e^{-i\frac
{\tau}{2}\HHH}\varphi\right|^2-\frac{2(N+2)}{N^2}\re
\int_{-\frac{\pi}{2}}^{\frac{\pi}{2}}\int_{\RR^N} G_0^{\frac
4N}e^{iN\tau}\left(e^{-i\frac{\tau}{2}\HHH}\varphi\right)^2.
\end{multline}

\subsection{Reduction of the problem}
\label{SS:reduction}

We prove here the following proposition:
\begin{prop}
 \label{P:reduction}
Assume that the conclusion of Theorem \ref{T:ortho} does not
hold. Then there exists an eigenfunction $\varphi$ of $\HHH$,
satisfying the orthogonality relations \eqref{ortho_intro} and such that
$Q(\varphi)=0$.
\end{prop}

We define
$$
E=\{\varphi\in L^2,\quad Q(\varphi)=0\}.
$$
Since $Q$ is a real positive quadratic form, we know that $E$ is a real vector
space. Before proving Proposition \ref{P:reduction}, we need a few
preliminary results.

\begin{lemma}
\label{L:compactQ}
Let $\{\varphi_{n}\}$ be a \emph{bounded} sequence in $L^2$ such that
\begin{equation}
 \label{nullQ}
\lim_{n\to \infty}Q(\varphi_{n})=0.
\end{equation}
Then there exists a subsequence of $\{\varphi_{n}\}$ that converges
strongly in $L^2$ to an element of $E$.
\end{lemma}
\begin{proof}
Assume after extraction, 
$$\varphi_{n} \rightharpoonup\varphi \text{ weakly in }L^2\text{ as }n\to\infty.$$
Write
\begin{equation}
\label{Q_B}
Q(\varphi)=c_Q\int |\varphi|^2+B(\varphi,\varphi),
\end{equation}
where $c_Q=C_S\frac{N+2}{N}$ and the symmetric bilinear form $B$ is defined by
\begin{multline*}
B(\varphi,\psi)=C_S\frac{4(N+2)}{N^2}\left(\re \int G_0\varphi\right)\left(\re \int G_0\psi\right)\\
-\frac{(N+2)^2}{N^2}\re \int_{\RR}\int_{\RR^N} |G|^{\frac{4}{N}}\left(e^{i\frac t2\Delta}\varphi \right)\left(e^{-i\frac t2\Delta}\overline{\psi}\right)\\
-\frac{2(N+2)}{N^2}\re \int_{\RR}\int_{\RR^N} G^{\frac 4N}\left(e^{i\frac t2\Delta}\varphi\right)\left(e^{i\frac t2\Delta}\psi\right).
\end{multline*}

We will use the following standard property of the Schr\"odinger linear flow:
\begin{claim}
\label{C:LS_compact}
$$\psi_n\rightharpoonup 0\text{ weakly in }L^2\Longrightarrow e^{i\frac{t}{2}\Delta}\psi_n\to 0\text{ strongly in }L^2_{\mathrm{loc}}(\RR\times \RR^N).$$
\end{claim}
Indeed, by the local smoothing effect \cite{Sj87,Ve88,CoSa89}, $e^{i\frac{t}{2}\Delta}$ defines a continuous map from $L^2(\RR^N)$ to $L^2\left(\RR,H^{1/2}_{\mathrm{loc}}(\RR^N)\right)$. Using the equation \eqref{CPlin}, we see that it also defines a continuous map from $L^2(\RR^N)$ to $H^{1/4}_{\mathrm{loc}}\left(\RR^{N+1}\right)$. The claim follows from the local compactness of the embedding of $H^{1/4}$ in $L^2$.

Combining Claim \ref{C:LS_compact} with the decay of $G$ at infinity, we get 
\begin{equation}
\label{compactnessB}
 \psi_n\rightharpoonup 0\text{ weakly in }L^2\Longrightarrow B(\psi_n,\psi_n)\to 0.
\end{equation}
We will show by contradiction that $\{\varphi_n\}$ is a Cauchy sequence in $L^2$.
If not, there exist sequences of integer $\{j_n\}$, $\{k_n\}$ that go to $\infty$ such that
\begin{equation}
\label{nonCauchy}
 \forall n,\quad \|\varphi_{k_n}-\varphi_{j_n}\|_{L^2}\geq \eps_0>0.
\end{equation}
The weak convergence of $\{\varphi_n\}$ in $L^2$ implies
\begin{equation}
\label{weak_kn}
\varphi_{k_n}-\varphi_{j_n} \rightharpoonup 0\text{ weakly in }L^2.
\end{equation} 
Furthermore, \eqref{nullQ} and Cauchy-Schwarz inequality ($Q$ is positive) implies
$$ 0\leq Q(\varphi_{j_n}-\varphi_{k_n})\leq 2\left(Q(\varphi_{j_n})+Q(\varphi_{k_n})\right)\longrightarrow 0\text{ as }n\to\infty.$$
Combining with \eqref{compactnessB} and \eqref{weak_kn} one gets
$$ \lim_{n\to \infty}\|\varphi_{j_n}-\varphi_{k_n}\|_{L^2}=0,$$
contradicting \eqref{nonCauchy}. The proof is complete.
\end{proof}

\begin{lemma}
\label{L:complex} The space $E$ is a finite dimensional vector space over $\CC$.
\end{lemma}
\begin{proof}
The space $E$ is a vector space over $\RR$. To  show that it is a vector space over $\CC$, it is sufficient to show that it is stable by multiplication by $i$.
Let $\varphi\in E$. Write $\varphi=\alpha G_0+\tilde{\varphi}$, with
$\alpha=\int \varphi G_0$, so that
\begin{equation}
 \label{ortho_tphi}
\int \tilde{\varphi} G_0=0.
\end{equation}
The function $i\alpha G$ is in $E$ and $E$ is stable by addition. To
show that $i\varphi\in E$ we must show that $i\tilde{\varphi}\in E$.
By \eqref{defQH},
\begin{multline*}
Q(i\tilde{\varphi})= Q(\tilde{\varphi})-\frac{8(N+2)}{N^2}
C_S\left(\re \int G_0\tilde{\varphi}\right)^2\\
+ \frac{4(N+2)}{N} \re \int_{-\pi/2}^{\pi/2}\int_{\RR^N} G_0^{\frac
4N}e^{iN\tau}\left(e^{-i\frac
{\tau}{2}\HHH}\tilde{\varphi}\right)^2.
\end{multline*}
We know that $\tilde{\varphi}\in E$, so $Q(\tilde{\varphi})=0$
and it suffices to show:
\begin{gather}
\label{null1}
\left(\re \int G_0\tilde{\varphi}\right)^2=0\\
\label{null2}
\re \int_{-\pi/2}^{\pi/2}\int_{\RR^N} G_0^{\frac 4N}e^{iN\tau}
\left(e^{-i\frac \tau 2\HHH}\tilde{\varphi}\right)^2=0.
\end{gather}
The first equality follows immediately from \eqref{ortho_tphi}. Let
us show the second equality in the case $N=2$. By
\eqref{ortho_tphi}, $\tilde{\varphi}$ is orthogonal to the first
eigenfunction $h_{00}$ of $\HHH$. Thus,
$e^{-i\frac{\tau}{2}\HHH}\tilde{\varphi}$ is of the form
$$
e^{-i\frac{\tau}{2}\HHH}\tilde{\varphi}=\sum_{\substack{(n_1,n_2)\in
\NN^2 \\n_1+n_2\geq 1}} \alpha_{n_1n_2}
e^{-i\tau(n_1+n_2+1)}h_{n_1n_2}(y),
$$
where by definition
$
\alpha_{n_1n_2}=\int_{\RR^2} \tilde{\varphi}(y)h_{n_1n_2}(y)dy.
$
It follows from the definition of $h_{n_1n_2}$ that it is even if
$n_1+n_2$ is even and odd if $n_1+n_2$ is odd. Expanding
$\left(e^{-i\frac{\tau}{2}\HHH}\tilde{\varphi}\right)^2$, we can write
$$\re \int_{-\pi/2}^{\pi/2}\int_{\RR^2} G_0^{2}e^{2i\tau}
\left(e^{-i\frac \tau 2\HHH}\tilde{\varphi}\right)^2 = \re \int_{-\pi/2}^{\pi/2}\int_{\RR^2} G_0^{2}e^{2i\tau} \sum_{m\geq 4}e^{-i\tau
m}g_m(y)\,dy\,d\tau,$$
where $m\geq 4$ and $g_m\in C^{\infty}\left(\RR^N\right)$ is 
exponentially decaying. Again, $g_m$
is even if $m$ is even and odd if $m$ is odd. Then \eqref{null2}
will follow from
\begin{equation}
\label{null2_bis}
\re \int_{-\pi/2}^{\pi/2}\int_{\RR^2} G_0^{2}(x)e^{i 2\tau} e^{-i\tau m}g_m(y)\,dy\,d\tau=0.
\end{equation}
We distinguish two cases. If $m$ is odd, then $
\int_{\RR^2}G_0(y)^2g_m(y) \,dy=0$ (it is the integral of an odd
function on $\RR^2$), and \eqref{null2_bis} follows. If $m$ is even,
using that $m\geq 4$, we get that $\int_{-\pi/2}^{\pi/2}
e^{2i\tau-i\tau m}\,d\tau=0,$ which implies also
\eqref{null2_bis}. This completes the proof of \eqref{null2} in the
case $N=2$. To prove \eqref{null2} in the case $N=1$ write
$$
e^{-i\frac{\tau}{2}\HHH}\tilde{\varphi}_0=\sum_{n\geq 1}
\alpha_{n} e^{-i\tau\left(n+\frac{1}{2}\right)}h_{n}(y),
$$
and argue as above. We leave the details to the reader.

It follows immediately from Lemma \ref{L:compactQ} that the unit ball
of $\left(E,\|\cdot\|_{L^2}\right)$ is compact, concluding the proof of Lemma \ref{L:complex}.
\end{proof}

We next prove Proposition \ref{P:reduction}. Let
$\widetilde{E}=F^{\bot}\cap E$. By definition, $\widetilde{E}$ is
the subspace of functions $\varphi\in L^2$ satisfying $Q(\varphi)=0$
and the orthogonality relations \eqref{ortho_intro}. By Lemma
\ref{L:complex} it is a complex,
finite dimensional vector space.

We argue by contradiction, assuming that the conclusion of
Theorem \ref{T:ortho} does not hold.

\noindent\emph{Step 1. Existence of a nontrivial null-space for
$Q$.} In this step we show that the negation of Theorem
\ref{T:ortho} implies that $\widetilde{E}$ is not reduced to
$\{0\}$. Indeed, in this case, there exists a sequence $\varphi_{n}$
in $L^2$ such that
\begin{equation}
 \label{contraQ}
\forall n,\quad \varphi_{n}\in F^{\bot}\text{ and
}nQ(\varphi_{n})<\left\|\varphi_{n}\right\|_{L^2}=1.
\end{equation}
By Lemma \ref{L:compactQ}, a subsequence of
$\left\{\varphi_{n}\right\}_n$ converges strongly in $L^2$ to some $\psi\in E$. The
condition $\left\|\varphi_{n}\right\|_{L^2}=1$ implies that
$\|\psi\|_{L^2}=1$ and, in particular, that $\psi\neq 0$. Furthermore,
$\varphi_{n}\in F^{\bot}$ for all $n$ and $F^{\bot}$ is closed,
thus, $\psi\in F^{\bot}$, which shows as announced that $\dim
\widetilde{E}\geq 1$.
\medskip

\noindent\emph{Step 2. Stability by the harmonic evolution.}

In this step we show that $\widetilde{E}$ is invariant by
$e^{-i\frac{\tau_0}{2}\HHH}$ for any $\tau_0\in \RR$. As
$\widetilde{E}$ is a complex vector space, it is equivalent to show
that $\widetilde{E}$ is invariant by $S(t_0)=e^{-i\frac{\HHH-N}{2}
\tau_0}$. The space $F$ admits a basis of eigenfunctions of
$\HHH$, thus $F^{\bot}$ is stable by $S(\tau_0)$. To prove that
$E$ is stable by $S(\tau_0)$, we rewrite the equation
\eqref{devt_Q} using the lens transform of \S \ref{SS:lens}
\begin{multline}
 \label{devt_QH}
C_S \left(\int_{\RR^N} |G_0+\varphi|^2\right)^{1+\frac{2}{N}}-
\int_{-\pi/2}^{\pi/2}\int_{\RR^N} \left|e^{-i\frac{N\tau}{2}}G_0
+e^{-i\frac{\tau}{2}\HHH}\varphi\right|^{2+\frac{4}{N}}\,dy\,d\tau\\
=Q(\varphi)+\OOO\left(\|\varphi\|_{L^2}^3\right).
\end{multline}
We will show that the two terms in the first line of \eqref{devt_QH}
do not change when replacing $\varphi$ by $S(\tau_0)\varphi$,
which will imply that
\begin{equation}
 \label{Qinvariant}
Q(S(\tau_0)\varphi)=Q(\varphi),
\end{equation}
and thus, that $E$ and $\widetilde{E}=E\cap F$ are stable by
$S(\tau_0)$).

By mass conservation
\begin{multline}
\int_{\RR^N} \left|G_0+S(\tau_0)\varphi\right|^2
=\int_{\RR^N} \left|e^{-i\frac{N\tau_0}{2}}G_0
+e^{-i\frac{\tau_0}{2}\HHH} \varphi\right|^2\\
=\int \left|e^{-i\frac{\tau_0}{2}\HHH}\left(G_0
+\varphi\right)\right|^2=\int \left|G_0+\varphi\right|^2.
\end{multline}
Similarly,
\begin{multline*}
\int_{-\pi/2}^{\pi/2}\int_{\RR^N} \left|e^{-i\frac{N\tau}{2}}G_0
+e^{-i\frac{\tau}{2}\HHH}S(\tau_0) \varphi\right|^{2+\frac{4}{N}}\\
=\int_{-\pi/2+\tau_0}^{\pi/2+\tau_0}\int_{\RR^N}\left|e^{-i\frac{\tau}{2}\HHH}
\left(G_0+\varphi\right)\right|^{2+\frac{4}{N}}=\int_{-\pi/2}^{\pi/2}\int_{\RR^N}
\left|e^{-i\frac{\tau}{2}\HHH}\left(G_0+\varphi\right)\right|^{2+\frac{4}{N}}.
\end{multline*}
The last equality is consequence of the following known identity
(see e.g. equality (2.5) in \cite{Carles09}), which can be easily
checked by expanding $\varphi$ in the Hilbert basis of $L^2$ given
by the eigenfunctions of $\HHH$:
$$
e^{-i\frac{\pi+\tau}{2}\HHH}\varphi(y)= e^{-iN\frac{\pi}{2}}
e^{-i\frac{\tau}{2}\HHH}\varphi(-y).
$$
This concludes the proof of \eqref{Qinvariant}.

\medskip

\noindent\emph{Step 3. End of the proof.}

We have shown that $e^{-i\frac{\tau}{2}\HHH}$ is a strongly
continuous group of operators on the finite dimensional vector space
$\widetilde{E}$. As a consequence,
$e^{-i\frac{\tau}{2}\HHH}=e^{\tau A}$ for some $A\in
\LLL(\widetilde{E})$ (see for example \cite[Theorem 2.9
p.11]{EnNa00}).

Let $f\in \widetilde E$. Then
$$
\lim_{\tau\to 0} \frac{e^{-i\frac{\tau}{2}
\HHH}f-f}{\tau}=\lim_{\tau\to 0} \frac{e^{\tau
A}f-f}{\tau}=Af.
$$
This shows that $f$ is in the domain of $\HHH$ and that
$Af=-\frac{i}{2}\HHH f$. As a consequence, $\HHH=2iA$ is a
continuous linear operator on $\widetilde{E}$. Using that
$\widetilde{E}$ is finite dimensional, we deduce that $\HHH$ admits
an eigenfunction in $\widetilde{E}$, concluding the proof of
Proposition \ref{P:reduction}. \qed

From now on we treat each dimension separately.

\subsection{1D case}
\label{SS:1d}
In this case, the quadratic form is

\begin{multline}
\label{Q1d}
Q(\varphi) = \sqrt 3 \int|\varphi|^2 \, dy + \frac{4 \sqrt 3}{\sqrt \pi}
\left(\re\int e^{-y^2/2} \varphi(y) \,dy \right)^2\\
 -\frac{9}{\pi} \int_{-\frac \pi 2}^{\frac \pi 2}\int e^{-2y^2} \, \left|e^{-i\frac{\tau}{2}\HHH}\varphi\right|^2 \, dy\, d\tau - \frac{6}{\pi} \re \int_{-\frac \pi 2}^{\frac \pi 2}\int e^{-2y^2}
e^{i\tau}\,\left(e^{-i\frac{\tau}{2}\HHH}\varphi\right)^2\,dy\, d\tau.
\end{multline}

Recall that $h_0$ is the $0$th Hermite function (the eigenfunction
corresponding to $\lambda_0 = 1$), and  $e^{-i \frac{\tau}{2}
\mathcal{H}}h_0 = e^{-i\frac{\tau}{2}}\, e^{-y^2/2}$. Similarly,
$$
h_1(y) ={\sqrt 2} \,y e^{-y^2/2} \quad \rightsquigarrow \quad
e^{-i\frac{\tau}{2}\HHH}h_1(y)= {\sqrt 2} \, e^{-\frac32 i \tau}\, y \, e^{-y^2/2},
$$
and
$$
h_2(y) = \frac1{\sqrt 2}\,(2y^2-1) e^{-y^2/2} \quad \rightsquigarrow
\quad e^{-i\frac{\tau}{2}\HHH}h_2(y)= \frac1{\sqrt 2} \,e^{-\frac52 i \tau} \, (2y^2-1) \,
e^{-y^2/2},
$$
then it is easy to check that
$$Q(h_0)=Q(ih_0)=Q(h_1)=Q(ih_1)=Q(h_2)=Q(ih_2)=0.$$
Note that for the rest of $h_j$, $j \geq 3$, we have $e^{-i\frac{\tau}{2}\HHH}h_j=
e^{-i (2j+1)\frac{\tau}{2}} h_j$, and when computing the quadratic form
$Q(h_j)$, we obtain that by orthogonality of $\{h_j\}$ the second
term in \eqref{Q1d} is zero. Integration in $t$ over the full circle makes the
fourth term vanish, therefore producing
$$
Q(h_j)=\sqrt 3 \int |h_j(y)|^2 \, dy - 9 \int e^{-2y^2}
|h_j(y)|^2 \, dy.
$$
Since $e^{-2y^2}$ is dominated by $e^{-y^2}$, we estimate the second
term by
$$
\int e^{-y^2} |h_j(y)|^2 \, dy = \frac{(2j)!}{2^{2j} (j!)^2}\,
\sqrt\frac{\pi}{2},
$$
(see \cite[Lemma 2.1]{Wang08}). Then, using the following estimate for the central
binomial coefficient
\begin{equation}
 \label{E:centralbinom}
\ds \binom{2m}{m} \leq \frac{4^{m}}{\sqrt{3m+1}},  \quad  m \geq 1,
\end{equation}
we obtain
\begin{align*}
Q(h_j) & \geq \sqrt{3\pi}\left(1- 3 \sqrt \frac32 \frac{(2j)!}{2^{2j} (j!)^2}\, \right)\\
& \geq \sqrt{3\pi}\left(1- \frac{3 \sqrt 3}{\sqrt 2 \sqrt
{3j+1}} \right) > 0,
\end{align*}
for $j > 4$. Explicit computation shows that
$$
Q(h_3)= \frac{2\sqrt \pi}{3 \sqrt 3} 
\quad \text{for} \quad h_3(y)
= \frac1{\sqrt 3}(2y^3-3y)e^{-y^2/2}
$$
and
$$
Q(h_4)= \frac{8\sqrt \pi}{9 \sqrt 3} 
\quad \text{for} \quad h_4(y)
= \frac1{2 \sqrt 6}(4y^4-12y^2+3)e^{-y^2/2},
$$
concluding the proof that $Q(h_j)>0$ for all $j \geq 3$.

\subsection{2D case}
\label{SS:2d}
Recall from \S \ref{SS:harmonic} the definitions of the basis
$h_{jk}$ of eigenfunctions of $\HHH$. By definition $h_{jk}(y) =
h_j(y_1)h_k(y_2)$, where $\{h_j\}_{j\geq 0}$ is the orthogonal system in
$L^2(\cR)$ of eigenfunctions of the $1D$ harmonic oscillator. The function $h_{jk}$ corresponds to the eigenvalue
$\lambda_{m}$ with $m=j+k$, and $\lambda_{m}=2m+2=2(j+k)+2$. For a
fixed $m$ there are $m+1$ independent eigenfunctions $h_{jk} \equiv
h_{j, m-j}$, $0 \leq j \leq m$, corresponding to $\lambda_m$. The
space $F$ is exactly
$$
F=\vect_{\CC}\big\{h_{00},h_{01},h_{10},h_{02}+h_{20}\big\}.
$$
By Proposition \ref{P:reduction}, the proof of Theorem \ref{T:ortho} in $2D$ is reduced to the following:
\begin{prop}
 Assume that $N=2$. Then
\begin{gather}
\label{Qnonzero1}
\text{If }\alpha\neq \beta\text{ or }\gamma\neq 0,
\quad Q\left(\alpha h_{02}+\beta h_{20}+\gamma h_{11}\right)>0.\\
\label{Qnonzero2}
\text{If }m\geq 3\text{ and }\sum_{j=0}^{m} |\alpha_j|^2 \neq 0,
\text{ then }Q\left(\sum_{j=0}^{m} \alpha_j h_{j, m-j}\right) > 0.
\end{gather}
\end{prop}
\begin{proof}
Let $\varphi\in L^2$. By \eqref{defQH} with $N=2$, we have
\begin{multline}
Q(\varphi)=\int_{\RR^2} |\varphi|^2+2\left(\re \int G_0\varphi\right)^2\\
-4\int_{-\pi/2}^{\pi/2}\int_{\RR^2} G_0^2\left|e^{-i\frac \tau
2\HHH}\varphi\right|^2 -2\re \int_{-\pi/2}^{\pi/2}\int_{\RR^2}
G_0^{2}e^{i2\tau}\left(e^{-i\frac \tau 2\HHH} \varphi\right)^2.
\end{multline}
It is easy to check that $Q(h_{00})=0.$

Let $m\geq 1$. Any eigenfunction of $\HHH$ for the eigenvalue $2m+2$ is of the form
\begin{equation}
\label{eigenfunction}
\varphi=\sum_{j=0}^{m} \alpha_j h_{j, m-j}.
\end{equation}
If $\varphi$ is of this form, then the second integral in  $Q(\varphi)$ 
vanishes because of the orthogonality of the $h_{jk}$'s and so does
the last term, since $\int_{-\pi/2}^{\pi/2} e^{i2mt} \, dt = 0$ as
$m\in \NN\setminus \{0\}$.

Recall that the first eigenfunction for $\HHH$ is
$h_{00}(y)=e^{-\frac 12|y|^2}$. Using that
$G_0=\frac{1}{\sqrt{\pi}}e^{-\frac{|y|^2}{2}}$, we obtain
\begin{equation}
\label{Q_on_ortho}
Q(\varphi)=\BBB(\varphi,\varphi),\quad
\BBB(\varphi,\psi)=\re \int_{\RR^2} \varphi\overline{\psi}\\
-4\re \int_{\RR^2} h_{00}^2\varphi\overline{\psi}.
\end{equation}

In particular, if $j+k\geq 1$,
\begin{align*}
Q(h_{jk})  = & \, \Big(\int h_j^2(y_1) \, dy_1
\Big) \Big( \int h_k^2(y_2) \, dy_2 \Big) \\
& - 4 \Big(\int e^{-y_1^2} \, h_j^2(y_1) \, dy_1 \Big)
\Big(\int e^{-y_2^2} \, h_k^2(y_2)\, dy_2 \Big) \\
& = \pi \Big(1- \frac{(2j)!(2k)!}{2^{2(j+k)-1} \, (j!)^2 \,
(k!)^2}\Big),
\end{align*}
where in the last line we used the product of Hermite functions from
\cite[Lemma 2.1]{Wang08}. As expected we get $Q(h_{01})=Q(h_{10})=0$.

Define
$$
G(j,k) = \left\{
\begin{array}{cl} \frac{(j+k)!}{2^{(j+k)-1/2} \,
\sqrt{j!} \, \sqrt{k!} \, \left( \frac{j+k}2\right)!} &
\text{for} \quad j+k - \text{even},\\
0 & \text{for} \quad j+k - \text{odd}.
\end{array}\right.
$$
For a product of two $G$ functions, write
$$
F(m,j,k) = G(j,k)G(m-j,m-k).
$$
Observe that $F$ is symmetric, i.e.,
$$
F(m,j,k) = F(m,k,j) = F(m,m-j,m-k)=F(m,m-k,m-j).
$$
Note as well that
$$
Q(h_{j, m-j}) = \pi \left(1 - F(m,j,j)\right),\quad j\neq
k\Longrightarrow \BBB\left(h_{j,m-j},h_{k,m-k}\right)=\pi F(m,j,k),
$$
and that for $\alpha, \beta, \gamma \in \cC$
$$
\frac{1}{\pi}Q(\alpha h_{02}+\beta h_{20}+ \gamma h_{11})= \frac14 |\alpha -
\beta|^2 + \frac12 |\gamma|^2,
$$
which is equal to zero if and only if $\alpha = \beta$ and
$\gamma=0$. This shows \eqref{Qnonzero1}.

Let us show \eqref{Qnonzero2}.

We have
\begin{align*}
\frac1{\pi}&\,Q\left(\sum_{j=0}^{m}\alpha_j h_{j, m-j}\right)\\
 &=
\sum_{j=0}^{m} |\alpha_j|^2 \left(1-F(m,j,j)\right) - 2 \left( \re
\!\!\!\sum_{\substack{j<k,\\j+k - even}}
\!\!\alpha_j \overline{\alpha}_k F(m, j,k) \right)\\
&\geq \sum_{j=0}^{m} |\alpha_j|^2  - \left( \sum_{j=0}^{m}
|\alpha_j|^2 \, F(m,j,j)+ \!\!\sum_{\substack{j<k,\\j+k - even}}
\!\!(|\alpha_j|^2+ |\alpha_k|^2) F(m, j,k) \right)\\
&\geq \sum_{j=0}^{m} |\alpha_j|^2 
\left(1- \sum_{\substack{k \in[0,m],\\j+k - even}}
F(m,j,k)\right),
\end{align*}
where we used the symmetry of $F$ in the last line. By
Cauchy-Schwarz, for any $j \in [0,m]$ we obtain
\begin{align*}
\mathcal{F}(m,j)&:=\!\!\sum_{\substack{k \in[0,m],\\j+k - even}}
\!\!\!F(m,j,k)\\
 &= \frac2{2^{2m}} \, \sum_{\substack{k \in[0,m],\\j+k
- even}}
\frac{(j+k)!(2m-(j+k))!}{\sqrt{j!k!(m-j)!(m-k)!}
\left(\frac{j+k}2\right)! \left(m- \frac{j+k}2\right)!}\\
& \leq\frac2{2^{2m}} \, \left( \sum_{\substack{k \in[0,m],\\j+k -
even}} 
\binom{j+k}{k}\binom{2m-(j+k)}{m-k}\right)^{1/2}\\
& \qquad \times \left(\sum_{\substack{k \in[0,m],\\j+k - even}}
\binom{j+k}{\frac{j+k}2}
\binom{2m-(j+k)}{m-\frac{j+k}2}\right)^{1/2}\\
& \leq \frac2{4^{m}} \,\,  {\rm I} \times {\rm II}.
\end{align*}

By elementary combinatorial arguments (see Appendix \ref{A:combinatorics}) and \eqref{E:centralbinom}, we estimate the term I 
\begin{align*}
{\rm I}^{\,2}  &
\leq \frac12 \left[ \binom{2m+1}{m+1}+ \binom{2m}{m} \right] 
= \frac12 \left[\frac{m+1}{2m+2} \binom{2m+2}{m+1} +
\binom{2m}{m}\right]\\
& < \frac12 \left[ \frac12 \frac{4^{m+1}}{\sqrt{3(m+1)+1}} +
\frac{4^m}{\sqrt{3m+1}} \right] \\
& = 4^m \left(\frac{1}{\sqrt{3m+4}} + \frac1{2\sqrt{3m+1}} \right).
\end{align*}
For the term II we use \eqref{E:centralbinom}, then decompose into
fractions:
\begin{align*}
{\rm II}^{\, 2} & \leq 4^m \sum_{\substack{k \in[0,m],\\j+k - even}}
\frac1{\sqrt{3 \,(\frac{j+k}2)+1}}\, \frac1{\sqrt{3 \, (m-\frac{j+k}2) +1}}\\
& = \frac{4^m}{\sqrt{3m+2}} \sum_{\substack{k \in[0,m],\\j+k -
even}} \left(\frac1{{3 \,(\frac{j+k}2)+1}} + \frac1{{3 \, (m-\frac{j+k}2) +1}} \right)^{1/2}.
\end{align*}
Using the inequality $\sqrt{a+b}\leq \sqrt{a}+\sqrt{b}$, reindexing the summation and estimating the sum we obtain
\begin{align*}
{\rm II}^{\, 2} & \leq \frac{4^m}{\sqrt{3m+2}} \!\sum_{l=0}^{[m/2]}
\left(\frac1{\sqrt{3l+1}} + \frac1{\sqrt{3(m- l)+1}} \right) \\ 
& = \frac{4^m}{\sqrt{3m+2}} \left(\sum_{l=0}^m \frac1{\sqrt{3l+1}}
+ \frac1{\sqrt{3 \frac{m}2+1}} \, \chi_{\{m-even\}} \right)\\
& \leq \frac{4^m}{\sqrt{3m+2}} \, \left(\frac23 \, (\sqrt{3m+1}-1)+1
+ \frac1{\sqrt{1.5m+1}}\, \chi_{\{m-even\}} \right),
\end{align*}
where $\chi_{\{m-even\}}=1$ if $m$ is even, $0$ if $m$ is odd.
Hence,
\begin{align*}
\mathcal{F}(m,j)
& \leq 2 
\left[\left(\frac{1}{\sqrt{3m+4}} +
\frac1{2\sqrt{3m+1}} \right) \right.\\
& \left. \qquad \times \frac{1}{\sqrt{3m+2}} \,
\left(\frac{2\sqrt{3m+1}+1}3 + \frac1{\sqrt{1.5m+1}}\,
\chi_{\{m-even\}} \right) \right]^{1/2}, \\
\end{align*}
which is less than 1 for $m \geq 7$. For $m=3,4,5,6$ we provide the
values of $\mathcal{F}(m,j)$ in Table \ref{T:small-m} (which
are all smaller than 1).
\begin{table}[h]
\begin{align*}
&\begin{tabular} [c]{|l||l|l|l|l|}
\hline $m=3$ & $\mathcal{F}(3,0)$&$\mathcal{F}(3,1)$&$\mathcal{F}(3,2)$&$\mathcal{F}(3,3)$\\
\hline  & $0.841$ & $0.591$ & $0.591$ & $0.841$\\
\hline
\end{tabular}\\
&\begin{tabular} [c]{|l||l|l|l|l|l|}
\hline $m=4$ & $\mathcal{F}(4,0)$&$\mathcal{F}(4,1)$&$\mathcal{F}(4,2)$&$\mathcal{F}(4,3)$&$\mathcal{F}(4,4)$\\
\hline  & $0.785$ & $0.5$ & $0.664$ & $0.5$ & $0.785$\\
\hline
\end{tabular}\\
&\begin{tabular} [c]{|l||l|l|l|l|l|l|} \hline $m=5$ &
$\mathcal{F}(5,0)$&$\mathcal{F}(5,1)$&$\mathcal{F}(5,2)$&$\mathcal{F}(5,3)$
&$\mathcal{F}(5,4)$&$\mathcal{F}(5,5)$\\
\hline  & $0.718$ & $0.492$ & $0.573$ & $0.573$ & $0.492$ & $0.718$\\
\hline
\end{tabular}\\
&\begin{tabular} [c]{|l||l|l|l|l|l|l|l|} \hline $m=6$ &
$\mathcal{F}(6,0)$&$\mathcal{F}(6,1)$&$\mathcal{F}(6,2)$&$\mathcal{F}(6,3)$
&$\mathcal{F}(6,4)$&$\mathcal{F}(6,5)$&$\mathcal{F}(6,6)$\\
\hline  & $0.673$ & $0.454$ & $0.563$ & $0.495$ & $0.563$ & $0.454$ & $0.673$\\
\hline
\end{tabular}
\end{align*}
\caption{Values of $\mathcal{F}(m,j)$ for $3 \leq m \leq 6$.}
 \label{T:small-m}
\end{table}
\end{proof}

\appendix

\section{Implicit function theorem and orthogonality conditions}
\label{A:implicit}
In this appendix we prove Claim \ref{C:implicit}.
By explicit computation,
\begin{equation}
\label{G0_identities}
\nabla G_0=-xG_0,\quad \Delta G_0=(|x|^2-N)G_0.
\end{equation}
The preceding identities imply that at the point $(0,\Gamma_{id}, G_0)$:
\begin{gather*}
\frac{\partial U_{\delta}}{\partial\theta_0}=-i\delta G_0,\qquad \frac{\partial U_{\delta}}{\partial\rho_0}=-\frac{N}{2}\delta G_0- \delta x\cdot \nabla G_0=-\frac{N}{2}\delta G_0+\delta|x|^2 G_0,\\
\frac{\partial U_{\delta}}{\partial \xi_0}=-i\delta xG_0,\qquad
\frac{\partial U_\delta}{\partial x_0}=-\nabla G_0= \delta xG_0\\
\frac{\partial U_\delta}{\partial t_0}=-\frac{i}{2} \delta\Delta G_0-i\gamma \delta^{\frac 4N+1}|G_0|^{\frac 4N}G_0=\frac{i}{2}\delta(N-|x|^2)G_0-i\gamma \delta^{\frac 4N+1}|G_0|^{\frac 4N}G_0.
\end{gather*}
Using the equalities
\begin{equation}
\label{norm_G0}
\int G_0^2=1,\quad \int |x|^2G_0^2=\frac{N}{2},\quad \int |x|^4G_0^2=\frac{N(N+2)}{4},
\end{equation}
which follow from the normalization of $G_0$ and \eqref{G0_identities}, we get
that the Jacobian $\begin{pmatrix} \frac{\partial \Phi_{\delta}^k}{\partial \theta_0},\frac{\partial \Phi_{\delta}^k}{\partial \rho_0},\frac{\partial \Phi_{\delta}^k}{\partial \xi_0},\frac{\partial \Phi_{\delta}^k}{\partial x_0},\frac{\partial \Phi_{\delta}^k}{\partial t_0} \end{pmatrix}_{k=1\ldots 5}$ of $\Phi_{\delta}$ with respect to the variables $(\theta_0,\rho_0,\xi_0,x_0,t_0)$ at the point $(0,\Gamma_{id}, G_0)$ is of the form
\begin{equation*}
 \begin{pmatrix}
  -1 & 0 & 0 & 0 & \frac {1 }{4}+\OOO\left(\delta^{4}\right)\\
0 & \frac{1}{2} &0&0&0\\
0 &0 & -\frac{1}{2} &0 & 0\\
0 & 0 & 0&\frac{1}{2}& 0\\
0& 0 & 0 & 0 & -\frac{1}{4}+\OOO\left(\delta^{4}\right)
 \end{pmatrix},\quad
\begin{pmatrix}
  -1 & 0 & 0 & 0 & 0 & 0 & \frac{1}{2}+\OOO\left(\delta^{2}\right)\\
0 & 1 &0&0&0&0&0\\
0 &0 & -\frac{1}{2}&0&0&0 & 0\\
0 &0 & 0&-\frac{1}{2}&0&0 & 0\\
0 & 0 & 0&0& \frac{1}{2}&0& 0\\
0 & 0 & 0&0& 0&\frac{1}{2}& 0\\
0 & 0 & 0 & 0 &0&0& -\frac 12+\OOO\left(\delta^{2}\right)
 \end{pmatrix}
\end{equation*}
in dimensions $N=1$ or $2$ respectively. Using that these matrices are invertible, and that their inverses may be estimated uniformly with respect to $\delta\in (0,\delta_0)$ ($\delta_0$ small), we deduce from the implicit functions theorem that there exists $\eps>0$ and a constant $C>0$ such that for small $\delta$, if $\|f-G_0\|_{L^2}<\eps$, there exists $\left(\theta_{\delta},\rho_{\delta},\xi_{\delta},x_{\delta},t_{\delta}\right)=\left(\theta_{\delta},\Gamma_{\delta}\right)$ such that
\begin{equation*}
|\theta_\delta|+|\rho_\delta-1|+|\xi_\delta|+|x_\delta|+|t_\delta|\leq C\|f-G_0\|_{L^2}\text{ and } \Phi_\delta(\theta_\delta,\Gamma_\delta,f)=0.
\end{equation*}
Applying this to the family $\left\{\delta^{-1} g_{\delta}\right\}_\delta$ of Step 1 in the proof of Proposition \ref{P:close_G}, we get as announced that there exists $(\theta_\delta,\Gamma_\delta)=(\theta_{\delta},\rho_{\delta},\xi_{\delta},x_{\delta},t_{\delta})$ such that
\begin{equation*}
\lim_{\delta\to \infty}|\theta_{\delta}|+|\rho_{\delta}-1|+|\xi_{\delta}|+|x_{\delta}|+|t_{\delta}|=0 \text{ and }\Phi_\delta\left(\theta_{\delta},\Gamma_\delta,\delta^{-1}g_{\delta}\right)=0,
\end{equation*}
concluding the proof.
\section{Constant in 1D and the generating function trick}
 \label{A:const1}
By \eqref{def_c*},
\begin{equation*}
D_1=6 \re \iint |G(t)|^{4} \overline{G}(t) r(t)\,dt\,dx,
\end{equation*}
where $r$ is the solution to
$$
i \partial_t r + \frac12 \Delta r + |G|^4 G = 0, \quad r(0,x)=0.
$$
Let $\LL$ be the lens transform defined in \S \ref{SS:lens}. By the
change of variable $t=\tan\tau$, $x=\frac{y}{\cos\tau}$,
$\tau\in (-\pi/2,\pi/2)$, we get
$$
D_1 = 6\re \int_{\RR}\int_{-\pi/2}^{\pi/2} |\LL
G|^4\,\overline{\LL G}\,\LL r\,d\tau\,dy.
$$
By the example at the end of \S \ref{SS:lens}, $\LL
G=\frac{1}{\pi^{1/4}}e^{-i\tau/2}e^{-y^2/2}$, and thus,
\begin{equation}
 \label{formula_C1*}
D_1 = \frac{6}{\pi^{5/4}}\re \int_{\RR}\int_{-\pi/2}^{\pi/2}
e^{-5y^2/2}e^{i\tau/2}\LL r\,d\tau\,dy.
\end{equation}
Denote $\tilde{r}=\LL r$. An explicit computation shows that
$\tilde{r}$ solves
$$
i\partial_{\tau}\tilde{r}-\frac{1}{2}\HHH\tilde{r}+ \frac{1}{\pi^{
\frac{5}{4}}}e^{-\frac{i}{2}\tau}e^{-\frac{5}{2}y^2}=0,\quad\tilde{r}(0,y)=0.
$$
By Duhamel's formula
\begin{equation}
 \label{duhamel_tr}
\tilde{r}(\tau,y) = \frac{i}{\pi^{5/4}} \, e^{-\frac{i \tau}{2}
\mathcal{H}} \int_0^{\tau} e^{-\frac{i \sigma}{2}} \, e^{\frac{i
\sigma}{2}\mathcal{H}} \left( e^{-\frac{5}{2} y^2} \right)\,d\sigma.
\end{equation}
Decompose
\begin{equation*}
e^{-\frac{5}{2}y^2} = \sum_{k\geq 0} \alpha_k h_k
\end{equation*}
with $\{h_k\}$'s as in \eqref{E:herm1} or \eqref{E:herm2}, and
$$
\mathcal{H} h_k = \lambda_k h_k \equiv (2k+1)\,h_k.
$$
Then the coefficients $\alpha_k$'s are given by
\begin{equation}
 \label{def_alphak}
\alpha_k = \frac1{\sqrt \pi}
\int_{-\infty}^{+\infty}e^{-\frac{5}{2}y^2}\, h_k(y)\,dy.
\end{equation}
Note that for $k$ odd, the eigenfunction $h_k$ is odd, and thus, the
corresponding coefficient $\alpha_k=0$. In the end of this appendix
we compute the rest of (even) coefficients using a generating
function trick of Wang \cite{Wang08} and obtain
\begin{equation}
 \label{E:alpha-2k}
\alpha_{2j} = (-1)^j \, \frac{\sqrt{(2j)!}}{3^j \sqrt{3}\, j!}.
\end{equation}
Since
$$
e^{\frac12 i \sigma \HHH} \left(e^{-\frac52 y^2}\right) = \alpha_0
e^{\frac12 i \sigma} h_0 + \sum_{k \geq 1} \alpha_k e^{i (k+\frac12)\sigma}
h_k,
$$
by \eqref{duhamel_tr} we have
\begin{equation}
 \label{E:r}
\tilde{r}(\tau,y) = \frac{i}{\pi^{5/4}} \, e^{-i \frac{\tau}{2}}
\left( \tau \alpha_0 \, \, h_0(y) - i  \sum_{k \geq 1}
\frac{\alpha_k}{k} (1-e^{-ik\tau}) h_k(y) \right).
\end{equation}
Substituting $\tilde r$ back into \eqref{formula_C1*}, we obtain
that the zeroth term from \eqref{E:r} vanishes when integrating in
$\tau$, and thus,
\begin{align*}
D_1 
&=\frac6{\pi^{\frac{5}{2}}}\re\int\int_{-\frac{\pi}{2}}^{\frac{\pi}{2}}
e^{-\frac52 y^2} \sum_{k \geq 1} \frac{\alpha_k}{k}
(1-e^{-ik\tau}) h_k(y) \, d\tau \, dy\\
& = \frac6{\pi^{\frac{5}{2}}} \sum_{k \geq 1} \frac{\alpha_k}{k} \re
\int_{-\frac{\pi}{2}}^{\frac{\pi}{2}} \left(1 -
e^{-ik\tau}\right)\, d\tau
\int e^{-\frac52 y^2} \, h_k(y)\, dy\\
& = \frac6{\pi^{\frac{5}{2}}} \sum_{k \geq 1} \frac{\alpha_k}{k}
\cdot \pi \cdot \sqrt\pi \alpha_{k},
\end{align*}
where we have used $\int_{-\pi/2}^{\pi/2}e^{-ik\tau}\,d\tau=0$
if $k$ is even, and $\alpha_k=0$ if $k$ is odd. By
\eqref{E:alpha-2k} and keeping only even terms ($k=2j$), we have
\begin{equation}
D_1
=\frac6{\pi} \sum_{j \geq 1} \frac{(\alpha_{2j})^2}{2j}= \frac1{\pi}
\sum_{j\geq1} \frac{(2j)!}{j \, 3^{2j}\,(j!)^2},
\end{equation}
and since $\sum_{j\geq1} \frac{(2j)!}{j \, 9^{j}\,(j!)^2} \approx \,
0.2724 $, we get
$$
D_1 = 
\frac1{\pi} \sum_{k\geq 1}\frac{(2k)!}{k\, 9^{k} \, (k!)^2} \approx
\frac1{\pi} \, 0.2724 \approx 0.0867.
$$

\medskip

\noindent\emph{Proof of \eqref{E:alpha-2k}}.

Here we compute coefficients of decomposition of $e^{-\frac52 x^2}$ in Hermite basis, adapting a method from \cite{Wang08}.
Recall the $k$-th Hermite polynomial $H_k$
$$
h_k(x)=\frac{H_k(x)}{\sqrt{2^k \, k!}}\, e^{-\frac{x^2}{2}}.
$$
We have
\begin{equation}
 \label{E:alpha}
\alpha_k = \frac{1}{\sqrt{2^k \, k! \, \pi}} \int_{-\infty}^{+\infty}H_k(x) \, e^{-3x^2}dx.
\end{equation}
Using the generating function representation
\begin{equation}
 \label{generatingfunction}
e^{2tx-t^2}=\sum_{n=0}^{+\infty}\frac{t^n}{n!} \, H_n(x),
\end{equation}
we observe that it is equivalent to
\begin{equation*}
e^{2\frac{t}{\sqrt{3}} \sqrt{3}x-\left(\frac{t}{\sqrt{3}}\right)^2}
\times e^{-\frac{2}{3}t^2}=\sum_{n=0}^{+\infty}\frac{t^n}{n!}H_n(x),
\end{equation*}
on the other hand, using \eqref{generatingfunction} again on the left side
\begin{equation*}
\sum_{j=0}^{+\infty}\frac{1}{j!}\left(\frac{t}{\sqrt{3}}\right)^j
H_j\left(\sqrt{3}x\right) \times \sum_{k=0}^{+\infty}
\frac{1}{k!}\left(-\frac{2}{3}\right)^k t^{2k} = \sum_{n=0}^{+\infty}\frac{t^n}{n!} \, H_n(x).
\end{equation*}
Expanding the product on the left-hand side and identifying the
powers of $t$, we get
\begin{align*}
\frac{1}{n!} \, H_n(x)
&= \sum_{j+2k=n} \left(-\frac23\right)^k \frac1{j!\,k!\,(\sqrt{3})^j} \, H_j \left(\sqrt{3}x\right)\\
& = \frac{1}{(\sqrt{3})^n}\sum_{j+2k=n}\frac{(-2)^k}{j!\,k!}H_j\left(\sqrt{3}x\right)
\end{align*}
Integrating both sides against $e^{-3x^2}$, we obtain
\begin{align*}
\frac1{n!} \int  H_n(x) \, e^{-3 x^2} \, dx
& = \frac{1}{(\sqrt{3})^n} \sum_{j+2k=n}\frac{(-2)^k}{j!\,k!} \int_{\cR}
H_j \left(\sqrt{3}x\right) \, e^{-3x^2}\, dx\\
& = \frac{1}{(\sqrt{3})^n} \sum_{j+2k=n}\frac{(-2)^k}{j!\,k!} \, \frac1{\sqrt 3} \int_{\cR}
H_0(y) \, H_j (y) \, e^{-y^2}\, dy.\\
\intertext{Thus by \eqref{E:herm3}}
\frac1{n!} \int  H_n(x) \, e^{-3 x^2} \, dx& = \frac{1}{(\sqrt{3})^{n+1}} \sum_{j+2k=n}\frac{(-2)^k}{j!\,k!} \, {\sqrt{2^j\, j!}} \,
\delta_{0j} \sqrt{\pi}\\
& =
\begin{cases}
 0 &\text{ if } n\text{ is odd};\\
\frac{(-2)^k \, \sqrt \pi}{(\sqrt{3})^{2k+1} \, k!} &\text{ if } n \text{ is even, }n=2k.
\end{cases}
\end{align*}
Thus,
$$
\int_{-\infty}^{\infty}H_{2k}(x) \, e^{-3 x^2} \, dx
= \frac{(2k)!}{k!} \, \frac{(-2)^k \, \sqrt \pi}{(\sqrt 3)^{2k+1}},
$$
which by \eqref{E:alpha} implies that
$$
\alpha_{2k} = (-1)^k \frac{\sqrt{(2k)!}}{3^k \sqrt 3 \,k!}.
$$

\section{Constant in 2D}
 \label{A:const2}

\begin{claim}
\begin{equation*}
D_2 = \frac1{2\pi} \ln
\frac43.
\end{equation*}
\end{claim}
\begin{proof}
Recall from \eqref{def_r} the definition of $r$.
By \eqref{def_c*} we must show
$$\re \iint |G|^2 \, \overline{G} \, r \, dt \, dx =\frac1{8\pi} \ln
\frac43.$$
We will prove this result by direct computation of the integral, which is essentially an integral of a Gaussian function (in $x$) and rational functions (in $s$ and $t$). 

By \eqref{def_G},
$$
r(t,x)=\frac{i}{\pi^{3/2}} \int_0^t \frac{1}{(1+is)(1+s^2)}\,
e^{ i\frac{(t-s)}{2}\Delta}\left(e^{-\frac{|x|^2}2 \,
\frac{(3-is)}{(1+s^2)}} \right) \, ds.
$$
Noting that 
\begin{equation*}
e^{i \frac{t}2 \Delta}\left(e^{-\alpha |x|^2}\right)
=\frac{1}{(1+2\alpha it)^{N/2}}\,e^{-\frac{\alpha |x|^2}{1+2\alpha it}}, \quad \re \alpha>0,
\end{equation*}
we get
$$
r(t,x) = \frac{i}{\pi^{3/2}} \int_0^t
\frac{1}{(1+is)(1+s^2+(s+3i)(t-s))} \, e^{-\frac{|x|^2}2
\frac{3-is}{1+s^2+(s+3i)(t-s)}} \, ds.
$$
Let 
$$A=1+s^2+(s+3i)(t-s)=1+st+3i(t-s),\qquad B=\frac{1}{2}\left(\frac{2}{1+t^2}+\frac{1}{1-it}+\frac{3-is}{A}\right).$$
Thus
$$|G|^2\,\overline{G}\,r=\frac{i}{\pi^3}\int_0^t \frac{1}{(1+t^2)(1-it)(1+is)A}e^{-|x|^2B}\,ds.$$
Integrating in space we obtain
\begin{align}
\label{interm}
\int_{\cR^2} |G|^2 \, \overline{G} \, r \, dx &=\frac{i}{\pi^2}\int_0^t \frac{1}{(1+t^2)(1-it)(1+is)\,A\,B}\,ds\\
\notag
& =\frac{1}{\pi^2}\int_0^{t} \frac{i}{(1-it)(1+is)(3(1+st)+5i(t-s))}\,ds.
\end{align}
By fraction decomposition with respect to the variable $s$,
\begin{multline*}
\re\left[\frac i{(1-it)(1+is)(3(1+st)+5i(t-s))}\right]\\
= \re\left[\frac{1}{8(1+t^2)}\left(\frac{i}{1+is}+\frac{5i-3t}{3+5it+(3t-5i)s}\right)\right]\\
=\frac{1}{8(1+t^2)}\left(\frac{s}{1+s^2}+\frac{25(t-s)-9t(1+ts)}{9(1+ts)^2+25(t-s)^2}\right).
\end{multline*}
Integrating with respect to the variable $s$ and coming back to \eqref{interm} we get:
\begin{equation*}
\operatorname{Re}\left( \int_{\cR^2} |G|^2 \, \overline{G}
\, r \, dx \right) 
= -\frac{1}{16\pi^2} \frac{\ln(1+t^2)+2\ln3 -
\ln(9+25t^2)}{1+t^2}.
\end{equation*}
Finally, we compute the space-time norm:
\begin{align*}
& \int_{-\infty}^\infty \operatorname{Re}\left(
\int_{\cR^2}
|G|^2 \, \overline{G} \, r \, dx \right) \, dt\\
& =
-\frac1{16\pi^2}\left(\int_{-\infty}^{\infty}\frac{\ln(1+t^2)}{1+t^2}\,dt
+2\ln3 \int_{-\infty }^{\infty} \frac{dt}{1+t^2} - \int_{-\infty
}^{\infty}\frac{\ln(9+25t^2)}{(1+t^2)}\,dt\right).
\end{align*}
We have $$\int_{-\infty}^{\infty}\frac{1}{\left( 1+t^{2}\right)}\,dt = \pi.$$
By the change of variable $t=\tan\tau$, $\tau\in (-\pi/2,\pi/2)$ and the classical formulas
$$ \int_0^{\frac{\pi}{2}}\ln(\cos\tau)\,d\tau=-\frac{\pi}{2}\ln 2,\quad \int_0^{\pi}\ln(a+b\cos\tau)\,d\tau=\pi\ln\left(\frac{a+\sqrt{a^2-b^2}}{2}\right),\; a>|b|,$$
one gets
\begin{equation*}
\int_{-\infty }^{\infty}\frac{\ln(9+25t^2)}{(1+t^2)} = 6\pi\ln 2,\quad
\int_{-\infty}^{\infty}\frac{  \ln\left(1+t^{2}\right)  }{\left(
1+t^{2}\right)  }dt= 2\pi \ln 2.
\end{equation*}
We leave the details of the computations to the reader. Combining the preceding equalities, we obtain as announced
\begin{equation*}
\int_{-\infty}^\infty \operatorname{Re}\left( \int_{\cR^2}
|G|^2 \, \overline{G} \, r \, dx \right) \, dt = \frac1{8\pi}
(\ln 4-\ln 3).
\end{equation*}
\end{proof}

\section{Bound of a sum of binomial coefficients}
\label{A:combinatorics}
Let $m\geq 1$ and $j\in \{0,\ldots,m\}$. In this appendix we sketch the proof of the following inequality
\begin{equation}
\label{E:combinatorics}
\sum_{\substack{k\in\{0,\ldots,m\}\\j+k \text{ even}}}\binom{j+k}{j}\binom{2m-(j+k)}{m-j}\leq \frac{1}{2}\binom{2m+1}{m+1}+\frac{1}{2}\binom{2m}{m}. 
\end{equation} 

For $n\in \NN^*$, let $I_{n}=\{1,\ldots, n\}$. Let $\PPP(I_{2m+1})$ be the set of all subsets of $I_{2m+1}$.
Define $O_{m,j}\subset \PPP(I_{2m+1})$ and $E_{m,j}\subset \PPP(I_{2m+1})$ as follows: a subset of $I_{2m+1}$ is in $O_{m,j}$ (respectively, $E_{m,j}$) if it has $m+1$ elements  $a_1<a_2<\ldots<a_{m+1}$ and if $a_{j+1}$ is odd (respectively, even). Then for fixed $j\in \{0,\ldots,m\}$,
$$ \left|O_{m,j}\right|=\sum_{\substack{k\in\{0,\ldots,m\}\\j+k \text{ even}}}\binom{j+k}{j}\binom{2m-(j+k)}{m-j} ,\quad \binom{2m+1}{m+1}=\left|O_{m,j}\right|+\left|E_{m,j}\right|.$$
Let us construct a one-to-one map $\Phi_j$ from $O_{m,j}$ to the disjoint union of $E_{m,j}$ and the set of $m$-elements subsets of $I_{2m}$. Let $S$ be a set which is in $O_{m,j}$, and $a_1<a_2<\ldots<a_{m+1}$ its $m+1$ elements. Then if $j\geq 1$ and $a_j<a_{j+1}-1$, or $j=0$ and $a_1>1$, we denote by $\Phi_j(S)$ the element of $E_{m,j}$ $\{a_1,\ldots,a_{j},a_{j+1}-1,a_{j+2},\ldots,a_{m+1}\}$ (i.e obtained from $S$ by shifting only the element $a_{j+1}$ to the left). If $a_j=a_{j+1}-1$, or $j=0$ and $a_1=1$, we denote by $\Phi_j(S)$ the subset  $\{a_1,\ldots,a_j,a_{j+2},\ldots,a_m\}$ of $I_{2m}$. The mapping $\Phi_j$ is clearly one-to-one: in the first case one can recover $S$ by shifting the $j+1$ element of $\Phi_j(S)$ to the right. In the second case, by adding to the set $\Phi_j(S)$ the element $b_{j}+1$ ($1$ if $j=0$), where $b_j$ is the $j$th element of $\Phi_j(S)$. Finally we obtain:
$$ |O_{m,j}|\leq |E_{m,j}|+\binom{2m}{m}\leq \binom{2m+1}{m+1}-|O_{m,j}|+\binom{2m}{m},$$
which yields \eqref{E:combinatorics}.

\bigskip

\bigskip

\bibliographystyle{alpha} 
\bibliography{toto}

\end{document}